\documentclass[opre,nonblindrev]{informs4_revised}

\OneAndAHalfSpacedXI

\usepackage{graphicx} % figure
\usepackage{bm} % math bold font
\usepackage{amstext} % text in math environment
\usepackage{subfig} % subfigure
\usepackage{algorithmic} % environment of algorithm
\usepackage{enumerate} % environment of enumerate
\usepackage{calc} % calculate the width of text
\usepackage{tabularx} % table
\usepackage{makecell} % table cell
\usepackage{multirow} % table multirow
\usepackage{xcolor} % revision display
\newif\ifshowrevisions
\showrevisionstrue

\usepackage[colorlinks=false, linkcolor=blue, filecolor=magenta]{hyperref} % hyperlink
\usepackage[ruled,linesnumbered]{algorithm2e} % algorithm

\usepackage{natbib}
 \bibpunct[, ]{(}{)}{,}{a}{}{,}%
 \def\bibfont{\small}%
\TheoremsNumberedThrough     % Preferred (Theorem 1, Lemma 1, Theorem 2)
\ECRepeatTheorems
\JOURNAL{Operations Research}

\EquationsNumberedThrough    % Default: (1), (2), ...
\MANUSCRIPTNO{}

\begin{document}
% Outcomment only when entries are known. Otherwise leave as is and
%   default values will be used.
%\setcounter{page}{1}
%\VOLUME{00}%
%\NO{0}%
%\MONTH{Xxxxx}% (month or a similar seasonal id)
%\YEAR{0000}% e.g., 2005
%\FIRSTPAGE{000}%
%\LASTPAGE{000}%
%\SHORTYEAR{00}% shortened year (two-digit)
%\ISSUE{0000} %
%\LONGFIRSTPAGE{0001} %
%\DOI{10.1287/xxxx.0000.0000}%

% Author's names for the running heads
% Sample depending on the number of authors;
% \RUNAUTHOR{Jones}
% \RUNAUTHOR{Jones and Wilson}
% \RUNAUTHOR{Jones, Miller, and Wilson}
% \RUNAUTHOR{Jones et al.} % for four or more authors
% Enter authors following the given pattern:
%\RUNAUTHOR{}

% Title or shortened title suitable for running heads. Sample:
% \RUNTITLE{Bundling Information Goods of Decreasing Value}
% Enter the (shortened) title:
% \RUNTITLE{Test}

\RUNTITLE{Connecting Extreme-Point Generation and Decision Rules in Two-Stage Distributionally Robust Optimization}

% ----------- Original --------------
% \TITLE{Objective Certification and Value Completion via Dual Extreme Points in Two-Stage Distributionally Robust Optimization}
% ----------- Revised ---------------
\TITLE{Connecting Extreme-Point Generation and Decision Rules in Two-Stage Distributionally Robust Optimization}

% % Block of authors and their affiliations starts here:
% % NOTE: Authors with same affiliation, if the order of authors allows,
% %   should be entered in ONE field, separated by a comma.
% %   \EMAIL field can be repeated if more than one author
% \ARTICLEAUTHORS{%
% \AUTHOR{First Author}
% %,\textsuperscript{a} Second Author,\textsuperscript{b} Third Author,\textsuperscript{c} Fourth Author,\textsuperscript{c}
\ARTICLEAUTHORS{%
\AUTHOR{Jin Qi\textsuperscript{1}, Chen Yang\textsuperscript{2}}
\AFF{\textsuperscript{1}Department of Industrial Engineering and Decision Analytics, Hong Kong University of Science and Technology, \EMAIL{jinqi@ust.hk}}
\AFF{\textsuperscript{2} School of Management, Zhejiang University, \EMAIL{chenyang12@zju.edu.cn}}
}
\ABSTRACT{Two-stage distributionally robust optimization chooses a here-and-now decision and a wait-and-see decision policy that adapts to uncertainty realizations. Decision-rule methods specify the form of this adaptive policy in advance, whereas decomposition-generation methods construct second-stage value information iteratively. We connect these approaches through extreme points of the second-stage dual problem. Each extreme point defines an affine value piece, and a compatible primal basis can define an affine policy piece. Solving the first-stage problem may require only a subset of these pieces. We specialize an extreme-point generation method to solve the first-stage problem. A separate linear-programming procedure adds pieces until it recovers the recourse value over the uncertainty set, and the pieces can also be used to recover the optimal recourse policy. Building on the algorithm output, we develop a posteriori exactness test for conventional decision rules. We give extensions for degeneracy, rank-deficient recourse, and structured random recourse. The proposed algorithm is computationally efficient on the reported instances, and the results show that completing the second-stage value across the uncertainty set requires more extreme points than solving the first-stage problem alone.}

% \AREAOFREVIEW{Optimization}

\maketitle

\section{Introduction}
Two-stage optimization separates a here-and-now decision from wait-and-see decisions chosen after uncertainty is observed. Applications include production and supply-chain allocation, transportation and network design, and energy investment and operations \citep{birge2011introduction,ruszczynski2003stochastic,georghiou2020primaldual}. Classical stochastic programming evaluates decisions under a specified distribution, whereas distributionally robust optimization (DRO) protects against a family of plausible distributions \citep{kuhn2024distributionallyrobustoptimization}. Stochastic programming can be regarded as a special case of DRO with a singleton distribution set. Both frameworks optimize the first-stage decision with an adaptive second-stage policy.

Unrestricted adaptive decisions are generally difficult to compute \citep{shapiro2005complexity,dyer2006computational}. Prescribed affine, lifted, finite-adaptability, and piecewise-affine rules yield finite-dimensional models by selecting a policy architecture in advance \citep{ben2004adjustable,bertsimas2019adaptive,bental2020tractable}. Their output includes an explicit mapping from uncertainty to recourse decisions, but exactness depends on whether the chosen architecture contains an optimal policy.

Decomposition and generation methods organize computation around a different object. They construct the recourse value from generated scenarios or dual information and, under their stated conditions, can recover an exact first-stage solution and objective \citep{vanslyke1969lshaped,zeng2013columnconstraint, wang2022secondorder}. These methods primarily target an optimal first-stage decision and objective value. They do not automatically yield both an explicit second-stage policy and the recourse value function.

% ----------- Original --------------
% The two approaches therefore make different information explicit. A prescribed rule supplies an implementable policy together with the objective attained within its class. A method based on decomposition and generation can recover an exact objective and first-stage decision but does not explicitly provide the general recourse value function or policy. Although both approaches are popular, their theoretical connection and practical implications are not well understood.
% ----------- Revised ---------------
The two approaches therefore make different information explicit. A prescribed rule supplies an implementable policy together with the objective attained within its class. A method based on decomposition and generation can recover an exact objective and first-stage decision but does not explicitly provide the general recourse value function or policy. Although both approaches are popular, there are gaps in their theoretical connection and practical implications.

We study how the dual extreme-point representation of the second-stage problem connects these two approaches. The analysis yields an iterative solution method and characterizes its relationship with decision-rule methods. An optimal objective and first-stage decision do not by themselves recover the recourse value and second-stage policy for every uncertainty realization. Recovering the value function and policy reveals the structure of the second-stage problem and provides a target against which prescribed decision-rule architectures can be assessed. An explicit recourse value and policy can also eliminate repeated online solution of the second-stage problem, which is useful when recourse must be computed frequently or involves a large model, as in power systems \citep{jiang2016explicit}.

Our main contributions can be summarized as follows.
\begin{itemize}
    \item We use the dual extreme-point representation to connect decomposition and generation methods with decision rules. Each dual extreme point defines an affine value piece, and a compatible primal basis supplies an affine policy piece.
    
    \item We prove that the optimal objective can be preserved by at most the affine dimension of the master domain. We then develop iterative algorithms that selectively generate objective-relevant pieces. Under exact separation, the objective-oriented algorithm returns the exact objective and an optimal first-stage decision in finitely many iterations.
    
    \item At a fixed first-stage decision, we develop a procedure that adds the missing pieces until it recovers the recourse value and an optimal second-stage policy for every uncertainty realization. The procedure also provides an a posteriori exactness test for piecewise-affine decision rules and affine decision rules with lifted uncertainty.
    
    \item Numerical experiments evaluate objective-oriented generation against existing decomposition and generation methods and separately study the additional pieces needed to recover the recourse value at reference first-stage decisions. The results support both computational performance and the gap between the pieces needed to solve the first-stage problem and those needed to reproduce the recourse value pointwise.
\end{itemize}

\paragraph{Notation.}
For a positive integer $n$, let $[n]:=\{1,\ldots,n\}$. Scalars are written in ordinary type $a$, vectors in bold lowercase type $\BFa$, and matrices in bold uppercase type $BFA$. We write $\mathbb{R}_+^n$ for the nonnegative orthant. For a finite set $\mathcal A$, $|\mathcal A|$ denotes its cardinality, and $\mathcal A^c = \mathcal B \setminus \mathcal A$ denotes its complement relative to the set $\mathcal B$. For an index set $\mathcal J$, $\bm{v}_{\mathcal J}$ denotes the corresponding subvector and $\bm{C}_{\mathcal J}$ the row submatrix indexed by $\mathcal J$. We use $\operatorname{supp}(\bm{v})$, $\operatorname{rank}(\bm{C})$, $\operatorname{range}(\bm{C})$, $\operatorname{ker}(\bm{C})$, $\operatorname{affdim}(\mathcal C)$, and $\operatorname{rec}(\mathcal Z)$ for support, rank, range, kernel, affine dimension, and recession cone, respectively. The uncertainty dimension is denoted by $N_\xi$, and $\mathcal M(\mathbb R^{N_\xi})$ denotes the set of probability measures on $\mathbb R^{N_\xi}$.

\section{Related Literature}
In this section, we review the literature on decision rules, multiparametric linear programming, and iterative methods. We then discuss how our work relates to these perspectives and summarize the connections in Table~\ref{tab:literature_positioning}.

\paragraph{Prescribed decision rules and exactness.}
Affine rules and richer decision rules are classical methods in adjustable robust optimization, stochastic programming, and DRO \citep{ben2004adjustable,kuhn2011primal,bertsimas2019adaptive,chen2008linear}. These methods directly parameterize an explicit policy. Liftings and piecewise-affine constructions enrich the mapping through additional coordinates or adding pieces \citep{georghiou2015generalized,bental2020tractable, thoma2026note}, while adaptive partitions vary a rule across selected regions \citep{bertsimas2015design}. Finite adaptability and $K$-adaptability instead use a finite collection of decisions or policies and choose among them after uncertainty is observed \citep{bertsimas2010finite,hanasusanto2015k,han2022finite}. Approximation bounds and structural exactness results identify when decision rules are effective and attain the optimum \citep{bertsimas2012power,bertsimas2010optimality,iancu2013supermodularity,georghiou2021optimality}. Unlike decision-rule methods, our framework does not prescribe a policy architecture; instead, it generates value pieces endogenously from the dual extreme-point representation. After recovering the optimal recourse value, we can test whether a prescribed architecture can reproduce a pointwise-optimal policy.

\paragraph{Exact second-stage structure and multiparametric linear programming.}
Classical multiparametric linear programming (LP) and stochastic-programming results establish that right-hand-side-parametric linear recourse has a piecewise-affine value function and admits piecewise-affine optimizer maps described by optimal bases and critical regions \citep{garstka1974decision,gal1972multiparametric,borrelli2003geometric}. They distinguish the value function from the optimizer correspondence and address degeneracy, under which regions may overlap and several bases or optimizers may be valid; lexicographic perturbation provides one refinement mechanism \citep{borrelli2003geometric,jones2007lexicographic}. An optimal basis induces an affine optimizer on a polyhedral critical region, and finitely many such regions yield complete piecewise-affine value and optimizer maps. We use this geometry to convert an exactly covering set of dual value pieces into pointwise-optimal second-stage decisions.

\paragraph{Exact decomposition and iterative generation.}
Exact iterative methods use the restricted-master and separation logic of Benders decomposition, the L-shaped method, and exchange methods \citep{benders1962partitioning,vanslyke1969lshaped,blankenship1976infinitely}. Depending on the formulation, they add uncertainty realizations, cuts, or dual information.
The C\&CG method of \cite{zeng2013columnconstraint} adds uncertainty realizations with scenario-specific second-stage variables. \cite{georghiou2020primaldual} transfer uncertainty-set vertices to a scenario set and construct an explicit piecewise-affine rule through simplicial decompositions and interpolation. \cite{song2015adaptive} refine partitions by grouping scenarios around shared optimal dual solutions. For continuous-support Wasserstein and optimal-transport models, \cite{duque2022distributionally} maintain support points together with scenario-indexed sets of dual extreme points and complete the dual information at stored points before separating new support points. \cite{gamboa2021decomposition} compare C\&CG with single-cut and multi-cut Benders variants for Wasserstein data-driven models, showing that one algorithmic framework can retain scenario copies, cuts, or both. \cite{byeon2025twostage} combine cutting planes with unified scenario generation for two-stage distributionally robust conic linear programs.

% ----------- Original --------------
% Other methods use dual extreme points of the second-stage problem. \cite{wang2022secondorder} generate selected dual extreme points iteratively and establish finite convergence under boundedness and exact solution of the separation subproblem, while their practical subproblem uses approximate algorithms. \cite{diazcachinero2026dualclustering} cluster scenarios that share an optimal dual solution and add only the required CVaR hyperplanes, retaining the dual information for reuse in finite-scenario stochastic programs. Our work shares the dual extreme-point generation perspective, but we distinguish the pieces required for objective certification from those required for value completion and policy recovery. This distinction connects the decision-rule and iterative-generation perspectives and provides theoretical insight into the structure of two-stage DRO. Moreover, regarding the model setting, both \cite{wang2022secondorder} and \cite{diazcachinero2026dualclustering} focus on specialized DRO classes involving second-order conic programs and CVaR, while our work focuses on two-stage DRO with general polyhedral uncertainty sets and linear recourse.
% ----------- Revised ---------------
Other methods use dual extreme points of the second-stage problem. \cite{wang2022secondorder} generate selected dual extreme points iteratively and establish finite convergence under boundedness and exact solution of the separation subproblem, while their practical subproblem uses approximate algorithms. \cite{diazcachinero2026dualclustering} cluster scenarios that share an optimal dual solution and add only the required CVaR hyperplanes, retaining the dual information for reuse in finite-scenario stochastic programs. Our work shares the dual extreme-point generation perspective, but we distinguish the pieces needed to solve the first-stage problem from the additional pieces needed to reproduce the recourse value at every realization and, under certain basis conditions, recover an optimal policy. This distinction connects the decision-rule and iterative-generation perspectives and provides theoretical insight into the structure of two-stage DRO. Moreover, regarding the model setting, both \cite{wang2022secondorder} and \cite{diazcachinero2026dualclustering} focus on specialized DRO classes involving second-order conic programs and CVaR, while our work focuses on two-stage DRO with general polyhedral uncertainty sets and linear recourse.

\begin{table}[htbp]
\centering
\caption{Position of this paper across the main literature streams}
\label{tab:literature_positioning}

\footnotesize
\renewcommand{\arraystretch}{1.08}
\begin{tabularx}{\dimexpr\linewidth-3pt\relax}{@{}>{\raggedright\arraybackslash}X >{\raggedright\arraybackslash}X >{\raggedright\arraybackslash}X >{\raggedright\arraybackslash}X@{}}
\hline
Literature stream & Representation & Main result or output & Relation to this paper \\
\hline

Prescribed decision rules and exactness (e.g., \cite{georghiou2021optimality})
& Preselected policy architectures
& Tractable or exact rules under specified structures
& We recover an optimal policy from dual extreme points and provide an optimality check for decision rules
 \\
\hline
Parametric LP (e.g., \cite{jones2007lexicographic})
& Value functions, bases, and critical regions
& Complete value and optimizer maps
& We use multiparametric LP geometry to derive our theoretical implication \\
\hline
Decomposition and generation methods (e.g., \cite{duque2022distributionally,wang2022secondorder})
& Cuts, support points, scenarios, or dual second-stage information
& Exact solutions under method-specific assumptions
& We share selective generation of dual information \\
\hline
This paper
& Dual extreme-point representation and compatible primal bases
& \multicolumn{2}{
    >{\raggedright\arraybackslash}
    p{\dimexpr.5\linewidth-1.5pt-\tabcolsep\relax}@{}
% ----------- Original --------------
%   }{Connection between decision-rule and iterative-generation perspectives.
%     Algorithms and implications for objective certification and value completion.}
% ----------- Revised ---------------
  }{Connection between decision-rule and iterative generation.
    Algorithms and implications for solving the first-stage problem and recovering recourse value information pointwise.}
\\
\hline
\end{tabularx}
\end{table}

\section{Model Setup}

We first consider a fixed-recourse two-stage DRO problem in which uncertainty affects only the right-hand side of the second-stage constraints. Let $\BFx_1\in\mathbb{R}^{N_1}$ and $\BFx_2\in\mathbb{R}^{N_2}$ denote the first- and second-stage decisions, respectively. The uncertain vector $\BFxi$ is observed between the two stages, with support $\Xi \subseteq \mathbb{R}^{N_\xi}$, and $\mathcal{P}$ denotes the ambiguity set. The problem is
\begin{equation}
    \label{eq:first_stage_problem_continuous}
    \begin{aligned}
        \min_{\BFx_1}\quad
        & \BFc^\top\BFx_1+\sup_{\mathbb{P}\in\mathcal{P}}\mathbb{E}_{\mathbb{P}}[\psi(\BFx_1,\BFxi)]\\
        \text{s.t.}\quad & \BFA\BFx_1\geq\BFa,
    \end{aligned}
\end{equation}
where the second-stage problem, also called the recourse value function, is
\begin{equation}
    \label{eq:second_stage_problem_continuous}
    \begin{aligned}
        \psi(\BFx_1,\BFxi)=\min_{\BFx_2} \left\{ \BFb^\top\BFx_2 \mid \BFB\BFx_1+\BFC\BFx_2\geq\BFW\BFxi+\BFw \right\}.
    \end{aligned}
\end{equation}
Suppose $\BFC \in \mathbb{R}^{M\times N_2}$ and $M \geq N_2$, and let $L$ denote the number of ambiguity-dual constraint families. For many ambiguity sets, problem \eqref{eq:first_stage_problem_continuous} admits the following ambiguity-dual epigraph representation.
\begin{equation}
    \label{eq:two_stage_reformulation_general}
    \begin{aligned}
        \min_{\BFx_1,\BFtheta}\quad &\BFc^\top\BFx_1+\phi_0(\BFtheta)\\
        \text{s.t.}\quad
        &\phi_\ell(\BFtheta,\BFxi)\geq\psi(\BFx_1,\BFxi),
        &&\forall\BFxi\in\Xi,\ \ell\in[L],\\
        &\BFA\BFx_1\geq\BFa,\quad \BFtheta\in\Theta.
    \end{aligned}
\end{equation}
The ambiguity set $\mathcal{P}$ determines $\BFtheta$ and the functions $\phi_0,\phi_\ell$. The exact epigraph and separation requirements depend on the ambiguity model. We assume that the DRO problem has a finite optimal value and attains an optimizer, and that the uncertainty set $\Xi$ is compact.

\paragraph{Mean-absolute deviation ambiguity set.}
The MAD ambiguity set \cite{goh2010distributionally} is given by
\begin{equation}
    \label{eq:mad_ambiguity_set}
    \mathcal{P}=\left\{\mathbb{P}\in\mathcal{M}(\mathbb{R}^{N_\xi})\ \middle|\
        \mathbb{P}(\BFxi\in\Xi)=1,
        \mathbb{E}_{\mathbb{P}}(\BFxi)=\BFmu,
        \mathbb{E}_{\mathbb{P}}(|\BFxi-\BFmu|)\leq\BFsigma
    \right\}.
\end{equation}
The corresponding epigraph reformulation is
\begin{equation}
    \label{eq:two_stage_reformulation_mad}
    \begin{aligned}
        \min_{\BFx_1,\theta_0,\BFtheta_1,\BFtheta_2}\quad
        &\BFc^\top\BFx_1+\theta_0+\BFmu^\top\BFtheta_1+\BFsigma^\top\BFtheta_2\\
        \text{s.t.}\quad
        &\theta_0+\BFtheta_1^\top\BFxi+\BFtheta_2^\top|\BFxi-\BFmu|
        \geq\psi(\BFx_1,\BFxi), &&\forall\BFxi\in\Xi,\\
        &\BFA\BFx_1\geq\BFa,\quad \BFtheta_2\geq0.
    \end{aligned}
\end{equation}

\paragraph{Wasserstein ambiguity set.}
Let $\hat{\mathbb{P}}$ denote the empirical distribution on the samples $\{\hat{\BFxi}_i\}_{i\in[L]}$. The 1-Wasserstein ambiguity set is
\begin{equation}
    \label{eq:wasserstein_ambiguity_set}
    \mathcal{P}=\left\{\mathbb{P}\in\mathcal{M}(\mathbb{R}^{N_\xi})\ \middle|\
    \mathbb{P}(\BFxi\in\Xi)=1,\ W_1(\mathbb{P},\hat{\mathbb{P}})\leq\epsilon\right\}, \ \text{where}\ W_1(\mathbb{P},\hat{\mathbb{P}})
    :=\inf_{\pi\in\Pi(\mathbb{P},\hat{\mathbb{P}})}
    \int_{\Xi\times\Xi}\|\BFxi-\BFxi'\|\,d\pi(\BFxi,\BFxi').
\end{equation}
Here $\Pi(\mathbb{P},\hat{\mathbb{P}})$ denotes the set of couplings with marginals $\mathbb{P}$ and $\hat{\mathbb{P}}$, and $\|\cdot\|$ denotes the ground norm. The corresponding epigraph form is
\begin{equation}
    \label{eq:two_stage_reformulation_wasserstein}
    \begin{aligned}
        \min_{\BFx_1,\theta_0,\ldots,\theta_L}\quad
        &\BFc^\top\BFx_1+\theta_0\epsilon+\frac{1}{L}\sum_{i\in[L]}\theta_i\\
        \text{s.t.}\quad
        &\theta_i+\theta_0\|\hat{\BFxi}_i-\BFxi\|\geq\psi(\BFx_1,\BFxi),
        &&\forall\BFxi\in\Xi,\ i\in[L],\\
        &\BFA\BFx_1\geq\BFa,\quad \theta_0\geq0.
    \end{aligned}
\end{equation}

\section{Fixed-Recourse Problem and Its Policy Geometry}
The continuous two-stage model is given in \eqref{eq:first_stage_problem_continuous}--\eqref{eq:second_stage_problem_continuous}. We now analyze the fixed-recourse structure underlying the main results. Define $\BFd(\BFxi, \BFx_1) = \BFW\BFxi + \BFw - \BFB\BFx_1$. The dual of the second-stage linear program is
\begin{equation}
    \label{eq:second_stage_problem_dual_continuous}
    \max_{\BFz \geq 0}\left\{\BFd(\BFxi, \BFx_1)^\top \BFz \mid \BFC^\top \BFz = \BFb\right\}.
\end{equation}

The fixed-recourse analysis relies on the fact that the recourse matrix $\BFC$ and the second-stage cost vector $\BFb$ are fixed. Consequently, the dual feasible polyhedron does not depend on $\BFx_1$ or $\BFxi$. We denote it by
\[
    \mathcal{Z}
    :=
    \{\BFz\mid \BFC^\top\BFz=\BFb,\ \BFz\geq\BFzero\}
\]
We assume that $\mathcal{Z}$ is nonempty.
As a polyhedron, $\mathcal{Z}$ can be represented by its extreme points and rays. Let $\mathcal{S}=\{\BFs_i\}_{i\in[S]}$ be the extreme-point set of $\mathcal{Z}$, where $S=|\mathcal{S}|$, and let $\mathcal{R}=\{\BFr_k\}_{k\in[R]}$ be the extreme-ray set of the recession cone \(
    \operatorname{rec}(\mathcal{Z})
    =
    \{\BFr\mid \BFC^\top\BFr=\BFzero,\ \BFr\geq\BFzero\}
\). The set $\mathcal{Z}$ then admits the point-and-ray representation
$$
    \mathcal{Z} = \left\{
        \BFz | \BFC^\top \BFz = \BFb, \BFz \geq 0
    \right\} = \left\{
        \sum_{i\in [S]} \lambda_i \BFs_i + \sum_{k \in [R]} \eta_k \BFr_k | \lambda_i \geq 0, \eta_k \geq 0, \sum_{i\in [S]} \lambda_i = 1
    \right\}.
$$

The point-and-ray representation allows us to express the dual problem \eqref{eq:second_stage_problem_dual_continuous} without explicit constraints. To ensure second-stage feasibility, Farkas' lemma requires the extreme-ray conditions $\BFd(\BFxi, \BFx_1)^\top \BFr_k \leq 0$ for all $\BFr_k\in\mathcal{R}$. We therefore define the extensive feasible region of $\BFx_1$ as
\[
    \mathcal{X}_1:=\left\{\BFx_1:\BFA\BFx_1\geq\BFa,\ 
    \sup_{\BFxi\in\Xi}\BFd(\BFxi,\BFx_1)^\top\BFr\leq0,
    \ \forall\BFr\in\mathcal{R}\right\}.
\]
The supremum imposes each ray inequality for every realization in $\Xi$. Thus, $\mathcal{X}_1$ retains the original first-stage constraints and excludes any first-stage decision for which an extreme ray detects an infeasible second-stage problem. Under relatively complete recourse, the inequalities $\sup_{\BFxi\in\Xi}\BFd(\BFxi,\BFx_1)^\top\BFr\leq0$ hold automatically.

For $\BFx_1\in\mathcal{X}_1$, finite LP duality yields the following pointwise-maximum representation of the second-stage value function.
\begin{equation}
    \label{eq:extreme_point_representation_continuous}
    \psi(\BFx_1, \BFxi) = \max_{\BFz \geq 0}\left\{\BFd(\BFxi, \BFx_1)^\top \BFz \mid \BFC^\top \BFz = \BFb\right\} = \max_{\BFs \in \mathcal{S}}\ \BFd(\BFxi, \BFx_1)^\top \BFs.
\end{equation}
% ----------- Original --------------
% Substituting \eqref{eq:extreme_point_representation_continuous} into \eqref{eq:two_stage_reformulation_general} yields the exact reformulation with dual extreme points
% ----------- Revised ---------------
Substituting \eqref{eq:extreme_point_representation_continuous} into \eqref{eq:two_stage_reformulation_general} yields the following exact reformulation in terms of dual extreme points.
\begin{equation}
    \label{eq:two_stage_reformulation_extreme_points_continuous}
    \begin{aligned}
        \min_{\BFx_1, \BFtheta}\ & \BFc^\top \BFx_1 + \phi_0(\BFtheta)
         \\
        \text{s.t.}\ &\phi_\ell(\BFtheta,\BFxi) \geq \BFd(\BFxi, \BFx_1)^\top \BFs, \qquad \forall \BFxi \in \Xi,\ \BFs \in \mathcal{S},\ \ell \in [L], \\
        & \BFx_1 \in \mathcal{X}_1, \BFtheta \in \Theta.
    \end{aligned}
\end{equation}
Because the dual feasible region $\mathcal{Z}$ is fixed, its extreme-point set $\mathcal{S}$ is also fixed. Thus, once $\mathcal{S}$ is known, problem \eqref{eq:two_stage_reformulation_extreme_points_continuous} is a conventional robust optimization problem. The formulation also reveals a natural piecewise-affine structure in the second-stage problem.

\subsection{From Value Pieces to Policy Pieces} 
The representation \eqref{eq:extreme_point_representation_continuous} shows that each extreme point $\BFs_i$ defines an affine value piece $\BFd(\BFxi,\BFx_1)^\top\BFs_i$. We use complementary slackness to recover a compatible affine policy piece from each value piece. We first consider full-column-rank recourse and nondegenerate dual extreme points. Section~\ref{sec:degenerate_policy_recovery} treats dual degeneracy and rank-deficient recourse. We impose the following assumption.
\begin{assumption}[Nondegenerate policy recovery]
    \label{asmp:nondegenerate_policy_recovery}
    The fixed recourse matrix and all dual extreme points satisfy
    \[
        \operatorname{rank}(\BFC)=N_2,
        \qquad
        |\operatorname{supp}(\BFs_i)|=N_2,
        \quad \forall \BFs_i\in\mathcal{S}.
    \]
\end{assumption}
Assumption~\ref{asmp:nondegenerate_policy_recovery} is needed for policy recovery but not for the value-function reformulation \eqref{eq:two_stage_reformulation_extreme_points_continuous}. 
For any nonempty indexed set $\mathcal T\subseteq\mathcal S$, define the deterministic selector
\(
i_{\BFx_1}^{\mathcal T}(\BFxi)
:=\min\arg\max_{j:\,\BFs_j\in\mathcal T}
\BFd(\BFxi,\BFx_1)^\top\BFs_j.
\)
Then, we can obtain the optimal decision rule as follows.
\begin{proposition}
    \label{prop:optimal_decision_rule_continuous}
    Suppose Assumption~\ref{asmp:nondegenerate_policy_recovery} holds. Let $\BFx_1^*$ be an optimal first-stage decision. 
    % Define
    % \(
    %     \Xi_i(\BFx_1)
    %     :=
    %     \left\{
    %     \BFxi\in\Xi:
    %     \BFd(\BFxi,\BFx_1)^\top\BFs_i
    %     \geq
    %     \BFd(\BFxi,\BFx_1)^\top\BFs_j,
    %     \forall j\in[S]
    %     \right\}
    % \), and 
    Let $\mathcal{J}_i=\operatorname{supp}(\BFs_i)$ for each $\BFs_i\in\mathcal{S}$.
    The optimal second-stage decision rule is
    \begin{equation}
        \label{eq:optimal_decision_rule_continuous}
        \BFx_2(\BFxi)
        =
        \BFC_{\mathcal{J}_i}^{-1}
        \BFd_{\mathcal{J}_i}(\BFxi,\BFx_1^*),
        \qquad i=i_{\BFx_1^*}^{\mathcal S}(\BFxi).
    \end{equation}
\end{proposition}
\proof{Proof.}
For each $\BFs_i\in\mathcal{S}$, $\mathcal{J}_i=\operatorname{supp}(\BFs_i)$ denotes the index set of its nonzero components. Under Assumption~\ref{asmp:nondegenerate_policy_recovery}, the $N_2$ rows in $\BFC_{\mathcal{J}_i}$ are linearly independent and therefore form a unique nonsingular basis. If $\BFs_i$ is dual optimal at $(\BFxi,\BFx_1^*)$, complementary slackness requires every row in $\mathcal{J}_i$ to be binding. Therefore,
\[
    \BFC_{\mathcal{J}_i}\BFx_2
    =
    \BFd_{\mathcal{J}_i}(\BFxi,\BFx_1^*)
    \Rightarrow
    \BFx_2^i(\BFxi)
    =
    \BFC_{\mathcal{J}_i}^{-1}
    \BFd_{\mathcal{J}_i}(\BFxi,\BFx_1^*).
\]
This policy is valid when $\BFs_i$ is dual optimal, and dual optimality implies primal feasibility. The selector chooses one such point at every $\BFxi\in\Xi$, so \eqref{eq:optimal_decision_rule_continuous} is feasible and optimal for the recourse problem \eqref{eq:second_stage_problem_continuous}.
\hfill \Halmos

Thus, Proposition~\ref{prop:optimal_decision_rule_continuous} shows that the dual extreme-point representation both provides an exact reformulation of the two-stage DRO problem and induces a natural piecewise-affine policy structure for the second-stage problem.

\subsection{Finite Extreme-Point Support for the Optimal Objective}
We have established the exact reformulation \eqref{eq:two_stage_reformulation_extreme_points_continuous} and, through Proposition~\ref{prop:optimal_decision_rule_continuous}, piecewise-affine policy recovery from the complete extreme-point set. A potential obstacle is that the number of dual extreme points $S$ can be very large, making both the reformulation and policy recovery impractical. We now show that only finitely many dual extreme points are needed to preserve the optimal objective value.

Classical support-reduction arguments for finite-dimensional semi-infinite convex programs follow from Helly's theorem \citep{shapiro2009semiinfinite,basu2017optimality}. We group the $L$ robust blocks of each dual vertex so that the resulting DRO-specific count applies to vertices.

To state the objective-support result, define
\(
    \mathcal{C}:=\mathcal{X}_1\times\Theta
\)
as the variable domain before the robust blocks are imposed. For each $(\BFs)\in\mathcal{S}$, define the feasible set induced by the corresponding robust block as
\[
    \mathcal{D}_{\BFs}
    :=
    \bigcap_{\ell=1}^L\mathcal{C}_{\BFs,\ell}, \text{ where }\mathcal{C}_{\BFs,\ell}
    :=
    \left\{
        (\BFx_1,\BFtheta)\in\mathcal{C}:
        \phi_\ell(\BFtheta,\BFxi)
        \geq
        \BFd(\BFxi,\BFx_1)^\top\BFs,
        \ \forall\BFxi\in\Xi
    \right\}.
\]
% For each $\BFs\in\mathcal{S}$, define the grouped vertex block
% \[
%     \mathcal{D}_{\BFs}
%     :=
%     \bigcap_{\ell=1}^L\mathcal{C}_{\BFs,\ell}.
% \]
The set $\mathcal{D}_{\BFs}$ imposes all robust constraints associated with vertex $\BFs$.
Let $p:=\operatorname{affdim}(\mathcal{C})$ denote the affine dimension of $\mathcal{C}$. When $\mathcal{C}$ is full-dimensional, $p=N_1+N_\theta$. For example, the MAD and Wasserstein reformulations \eqref{eq:two_stage_reformulation_mad} and \eqref{eq:two_stage_reformulation_wasserstein} have $N_\theta=1+2N_\xi$ and $N_\theta=L+1$. We impose the following assumption.
\begin{assumption}[Finite-dimensional convexity]
\label{asmp:support_reduction_regular}
The set $\mathcal{C}$ is nonempty and convex and has finite affine dimension $p$. The objective $\BFc^\top\BFx_1+\phi_0(\BFtheta)$ is convex on $\mathcal{C}$ and admits a finite optimal value. Each set $\mathcal{C}_{\BFs,\ell}$ is also convex.
\end{assumption}
The MAD and Wasserstein reformulations \eqref{eq:two_stage_reformulation_mad} and \eqref{eq:two_stage_reformulation_wasserstein} satisfy this assumption. The following theorem shows that at most $p$ dual extreme points can preserve the optimal objective value.

\begin{theorem}[Finite extreme-point support for the optimal objective]
\label{thm:size_extreme_point_set}
Suppose Assumption~\ref{asmp:support_reduction_regular} holds. There exists a subset $\mathcal{S}_{\mathrm{obj}}\subseteq\mathcal{S}$ with
\(
    |\mathcal{S}_{\mathrm{obj}}|\leq p,
\)
such that the master obtained by retaining
\[
    \phi_\ell(\BFtheta,\BFxi)
    \geq
    \BFd(\BFxi,\BFx_1)^\top\BFs,
    \qquad
    \forall \BFxi\in\Xi,\ \BFs\in\mathcal{S}_{\mathrm{obj}},\ \ell\in[L],
\]
has the same optimal objective value as the full master \eqref{eq:two_stage_reformulation_extreme_points_continuous}.
\end{theorem}

\proof{Proof.}
Let $v^*$ be the finite optimal value of the full master and define the strict lower-objective set
\[
    \mathcal{L}:=\{(\BFx_1,\BFtheta)\in\mathcal{C}:\BFc^\top\BFx_1+\phi_0(\BFtheta)<v^*\}.
\]
Assumption~\ref{asmp:support_reduction_regular} implies that this set is convex. In the $p$-dimensional affine hull of $\mathcal{C}$, the finite family
\(
    \{\mathcal{L}\}
    \cup
    \{\mathcal{D}_{\BFs}:\BFs\in\mathcal{S}\}
\)
has an empty intersection. Otherwise, a point in the intersection would be feasible for the full master and have an objective value below $v^*$.

Every subfamily consisting only of vertex blocks $\mathcal D_s$ has a nonempty intersection because a full-master feasible point belongs to every $\mathcal{D}_{\BFs}$. Hence, any empty-intersection subfamily supplied by Helly's theorem must contain $\mathcal{L}$, leaving at most $p$ vertex blocks. Let $\mathcal{S}_{\mathrm{obj}}$ index these blocks. Then $\mathcal{L}\cap\bigcap_{\BFs\in\mathcal{S}_{\mathrm{obj}}}\mathcal{D}_{\BFs}$ is empty and $|\mathcal{S}_{\mathrm{obj}}|\leq p$. The restricted master therefore has an objective value of at least $v^*$. Because it is a relaxation of the full master, its objective value is also at most $v^*$ and must equal $v^*$.
\hfill \Halmos
\endproof

Theorem~\ref{thm:size_extreme_point_set} establishes that at most $p$ extreme points can preserve the optimal objective value. The result is existential and does not provide a constructive selection procedure.
Moreover, Theorem~\ref{thm:size_extreme_point_set} concerns only the optimal objective value. It does not imply that the retained extreme points reproduce the recourse value function $\psi(\BFx_1^*,\BFxi)$ pointwise over $\Xi$ at a full-master optimizer $\BFx_1^*$. In particular, it is possible that
\[
    \exists\,\BFxi\in\Xi
    \quad\text{such that}\quad
    \max_{\BFs\in\mathcal{S}_{\mathrm{obj}}}
    \BFd(\BFxi,\BFx_1^*)^\top\BFs
    \neq
    \psi(\BFx_1^*,\BFxi).
\]
Additional pieces may therefore be required even after the optimal first-stage value has been found. 
% The next section develops one procedure for solving the first-stage problem and another for recovering the recourse value at every realization.

\subsection{Iterative Extreme-Point Generation Algorithm}
We design two extreme-point generation (EPG) phases with different targets. EPG for the objective (EPG-O) performs objective certification by identifying pieces sufficient to establish the optimal objective and first-stage decision. The EPG for the value function (EPG-V) fixes the EPG-O decision $\BFx_1^*$ and adds pieces until they reproduce the exact recourse value function $\psi(\BFx_1^*,\BFxi)$ pointwise over $\Xi$. Once EPG-V reproduces this value function, the recovery argument in Proposition~\ref{prop:optimal_decision_rule_continuous} yields an optimal second-stage policy.

The algorithm uses a classical master-subproblem framework. At each iteration, the master retains only the generated extreme points and therefore relaxes the full extreme-point reformulation \eqref{eq:two_stage_reformulation_extreme_points_continuous}. The subproblem searches over $\Xi$ and $\mathcal{S}$ for the largest violation among the constraints. This joint search is the principal computational challenge. We use approximate separation to discover promising pieces efficiently and retain exact global separation for certification. The generated pieces serve two distinct purposes. EPG-O certifies the optimal objective and first-stage decision, whereas EPG-V completes the recourse-value representation needed for policy recovery. Section~\ref{sec:numerical_experiments} evaluates the proposed algorithm numerically.
% A similar algorithmic design appears in \cite{wang2022secondorder} for a specific Wasserstein DRO reformulation with $\BFXi = \mathbb{R}^{N_\xi}$. Allowing a general polyhedral $\Xi$ creates a central computational challenge in the subproblem. We partially address this challenge with an approximate separation scheme that is particularly useful during early iterations. 

\subsubsection{Relaxed Master Problem}
We first define the relaxed master problem (RMP) solved at each EPG iteration. At iteration $t$, let $\mathcal{S}_t=\{\BFs_1,\ldots,\BFs_{S_t}\}$ denote the distinct generated vertices, where $S_t=|\mathcal{S}_t|$. The RMP retains the full region $\mathcal{X}_1\times\Theta$ but includes only the constraints indexed by $\mathcal{S}_t$. Let $\mathrm{RMP}_t$ denote its optimal value at iteration $t$. The RMP is
\begin{equation}
    \label{eq:rmp_continuous_exact}
    \begin{aligned}
        \mathrm{RMP}_t = \min_{\BFx_1, \BFtheta}\ & \BFc^\top \BFx_1 + \phi_0(\BFtheta)\\
        \text{s.t.}\ & \sup_{\BFxi \in \Xi}\left\{\BFd(\BFxi,\BFx_1)^\top\BFs_i - \phi_\ell(\BFtheta,\BFxi)\right\} \leq 0, \qquad \forall i \in [S_t],\ \ell \in [L], \\
        & \BFx_1 \in \mathcal{X}_1, \BFtheta \in \Theta.
    \end{aligned}
\end{equation}
Because $\mathcal{S}_t$ is much smaller than $\mathcal{S}$, $\mathrm{RMP}_t$ is a relaxation and provides a lower bound on the full problem \eqref{eq:two_stage_reformulation_extreme_points_continuous}. We therefore let $\mathrm{LB}_t := \mathrm{RMP}_t$ denote the lower bound at iteration $t$.

\subsubsection{Objective Separation}
Solving the relaxed master problem \eqref{eq:rmp_continuous_exact} at iteration $t$ yields the candidate $(\BFx_1^{(t)}, \BFtheta^{(t)})$. We next seek an extreme point whose cut removes this candidate by maximizing the violation of the full constraints in \eqref{eq:two_stage_reformulation_extreme_points_continuous}.

Because problem \eqref{eq:two_stage_reformulation_extreme_points_continuous} contains $L$ robust constraint families, it suffices to find the largest violation among them. For each $\ell\in[L]$, we first define the worst violation of the $\ell$-th constraint family and then maximize over $\ell$. The violation subproblem $V^{(t)}$ is
\begin{equation}
    \label{eq:violation_subproblem_continuous}
    V^{(t)} := \max_{\ell \in [L]} V^{(t)}_\ell, \quad V_\ell^{(t)} := \max_{\BFxi \in \Xi,\; \BFz \in \mathcal{Z}}\ \left[\BFd(\BFxi, \BFx_1^{(t)})^\top \BFz - \phi_\ell(\BFtheta^{(t)},\BFxi)\right].
\end{equation}
A positive $V^{(t)}$ identifies a cut that removes the current RMP optimizer and leads to a nondecreasing lower bound. A nonpositive value certifies that every extreme-point constraint is satisfied, but this certification requires a globally solved exact oracle to problem \eqref{eq:violation_subproblem_continuous}.

When $V^{(t)} > 0$, the violation subproblem also yields a natural upper bound on problem \eqref{eq:two_stage_reformulation_extreme_points_continuous}. We use $V^{(t)}$ to shift the ambiguity epigraph by a common intercept. Suppose there exists a direction $\BFeta$ such that, for every $\BFtheta\in\Theta$ and every $\delta\geq0$, $\BFtheta+\delta\BFeta\in\Theta$ and
\[
    \phi_\ell(\BFtheta+\delta\BFeta,\BFxi)
    \geq
    \phi_\ell(\BFtheta,\BFxi)+\delta,
    \qquad
    \forall \BFxi\in\Xi,\ \ell\in[L].
\]
For a candidate RMP solution, set
\(
    \delta_t:=\max\{0,V^{(t)}\},
    \bar{\BFtheta}^{(t)}:=\BFtheta^{(t)}+\delta_t\BFeta.
\)
The following proposition gives the resulting repair certificate.
\begin{proposition}
    \label{prop:continuous_intercept_repair_upper_bound}
    Let $(\BFx_1^{(t)},\BFtheta^{(t)})$ be an optimal solution of the relaxed master problem at iteration $t$, with lower bound $\mathrm{LB}_t=\mathrm{RMP}_t$. Suppose $\BFx_1^{(t)}\in\mathcal{X}_1$ and the intercept-shift property above holds. Then $(\BFx_1^{(t)},\bar{\BFtheta}^{(t)})$ is feasible for the master problem \eqref{eq:two_stage_reformulation_extreme_points_continuous}, and
    \[
        \mathrm{UB}_t
        =
        \BFc^\top \BFx_1^{(t)}+\phi_0(\bar{\BFtheta}^{(t)})
        =
        \mathrm{LB}_t+\phi_0(\bar{\BFtheta}^{(t)})-\phi_0(\BFtheta^{(t)})
    \]
    is a valid upper bound. If $V^{(t)}\leq0$, then the relaxed-master solution is optimal for the full master problem.
\end{proposition}
% ----------- Original --------------
% Consequently, the incumbent upper bound can be updated with this repaired solution whenever the shift property is available. 
% ----------- Revised ---------------
% Whenever the shift property is available, this repaired solution can therefore update the incumbent upper bound.

\proof{Proof.}
The definitions of $V^{(t)}$ and $\delta_t$ imply, for every $\BFxi\in\Xi$, $\BFz\in\mathcal{Z}$, and $\ell\in[L]$,
\[
    \phi_\ell(\bar{\BFtheta}^{(t)},\BFxi)
    \geq\phi_\ell(\BFtheta^{(t)},\BFxi)+\delta_t
    \geq\BFd(\BFxi,\BFx_1^{(t)})^\top\BFz.
\]
Together with $\BFx_1^{(t)}\in\mathcal{X}_1$, these inequalities establish full-master feasibility, and the upper-bound identity follows by substitution. If $V^{(t)}\leq0$, then $\delta_t=0$, so the feasible RMP optimizer attains its lower bound in the full master and is optimal.
\hfill \Halmos
\endproof

For the MAD reformulation \eqref{eq:two_stage_reformulation_mad}, this common-intercept condition holds with $\BFeta$ corresponding to the scalar component $\theta_0$. Increasing $\theta_0$ by $\delta_t$ shifts both $\theta_0+\BFtheta_1^\top\BFxi+\BFtheta_2^\top|\BFxi-\BFmu|$ and $\phi_0(\BFtheta)=\theta_0+\BFmu^\top\BFtheta_1+\BFsigma^\top\BFtheta_2$ upward by the same amount. For the Wasserstein reformulation \eqref{eq:two_stage_reformulation_wasserstein}, the analogous direction shifts the relevant sample-wise intercept coordinates simultaneously. In both cases, the repaired objective satisfies
\[
    \mathrm{UB}_t
    =\mathrm{LB}_t+\delta_t
    =\mathrm{LB}_t+\max\{0,V^{(t)}\}.
\]
\paragraph{Approximate Separation with Candidate Worst-Case Scenarios}
The exact violation problem \eqref{eq:violation_subproblem_continuous} is bilinear and difficult to solve. We obtain an approximation by restricting the scenario search to candidate worst-case scenarios generated by the current relaxed master problem \eqref{eq:rmp_continuous_exact}.

For each $\ell\in[L]$, let $\Xi^{(t)}_\ell$ denote the finite collection of worst-case scenarios, and let $\Xi^{(t)} = \bigcup_{\ell\in[L]}\Xi^{(t)}_\ell$ denote the union of all candidate worst-case scenarios.
These worst-case scenarios can be identified through complementary slackness in the dual reformulation of the robust counterpart.
% For example, if $\BFalpha_{i,\ell}^{(t)}$ is the optimal dual variable for the support constraint associated with $(\BFs_i,\ell)$ at iteration $t$, then the corresponding worst-case scenario $\BFxi_{i,\ell}^{(t)}$ satisfies $\BFalpha_{i,\ell}^{(t)\top}(\BFH\BFxi_{i,\ell}^{(t)} - \BFh) = 0$. Hence, we can obtain a candidate worst-case scenario by solving a linear equation of the basis matrix corresponding to $\BFalpha_{i,\ell}^{(t)}$.
% When adding a new objective piece, we want to identify a dual extreme point that violates one of the currently protected constraint families at one of these candidate scenarios. 
% ----------- Original --------------
% Using $\Xi^{(t)}$ as a candidate scenario set, we can define an approximate violation subproblem $V'^{(t)}$ as
% ----------- Revised ---------------
Using $\Xi^{(t)}$ as the candidate scenario set, define the approximate violation subproblem $V'^{(t)}$ as
\begin{equation}
    \label{eq:violation_subproblem_continuous_approx}
    V'^{(t)}
    :=
    \max_{\BFz \in \mathcal{Z}}
    \max_{\ell\in[L]}
    \max_{\BFxi \in \Xi^{(t)}} 
    \left\{
        \BFd(\BFxi,\BFx_1^{(t)})^\top \BFz - \phi_\ell(\BFtheta^{(t)},\BFxi)
    \right\}.
\end{equation}
The approximate subproblem $V'^{(t)}$ avoids bilinearity because $\Xi^{(t)}$ is finite. Clearly, $V'^{(t)} \leq V^{(t)}$. If $V'^{(t)} > 0$, we add an extreme-point optimizer to $\mathcal{S}_t$ as a discovery cut.
If $V'^{(t)} \leq 0$, however, we cannot conclude that $V^{(t)} \leq 0$. We then solve the exact violation subproblem $V^{(t)}$ to determine whether a violated constraint exists outside the candidate set. Because of its low computational cost, the approximate violation subproblem $V'^{(t)}$ is particularly useful in early iterations.

\subsubsection{Objective Certification and Value Completion Procedures}
\paragraph{EPG-O for Objective Certification.}
EPG-O iteratively solves the RMP and the objective-separation problem, generating extreme points until the RMP solution is certified as optimal for the full problem. Algorithm~\ref{alg:continuous_iterative} summarizes the procedure.
\par
% \begin{algorithm}[H]
% \caption{EPG-O: Extreme-point generation for objective certification}
% \label{alg:continuous_iterative}
% \KwIn{Problem data, a nonempty seed $\mathcal{S}_1\subseteq\mathcal{S}$, and $\varepsilon_{\mathrm{obj}}\geq0$}
% Initialize $t=1$.\;
% \While{true}{
% Solve RMP~\eqref{eq:rmp_continuous_exact} to obtain $(\BFx_1^{(t)},\BFtheta^{(t)})$.\;
% Solve the approximate violation problem~\eqref{eq:violation_subproblem_continuous_approx} and obtain $V'^{(t)}$. If $V'^{(t)} > 0$, select a maximizing dual solution as an extreme point $\BFs^\star$.\;
% Otherwise, solve the exact violation problem~\eqref{eq:violation_subproblem_continuous} and obtain an extreme point $\BFs^\star$.\;
% \If{$V^{(t)}\leq\varepsilon_{\mathrm{obj}}$}{
% Delete from $\mathcal{S}_t$ any vertices that do not affect the optimal value.\;
% \KwRet{$(\BFx_1^{(t)},\BFtheta^{(t)},\mathcal{S}_t)$}\;
% }
% $\mathcal{S}_{t+1}\leftarrow\mathcal{S}_t\cup\{\BFs^\star\}$; $t\leftarrow t+1$.\;
% }
% \end{algorithm}
\begin{algorithm}[H]
\caption{EPG-O for objective certification}
\label{alg:continuous_iterative}
\KwIn{Problem data, a nonempty seed $\mathcal{S}_1\subseteq\mathcal{S}$, and tolerances $\varepsilon_{\mathrm{disc}},\varepsilon_{\mathrm{obj}}\geq0$}
Initialize $t=1$.\;
\While{true}{
    Solve RMP~\eqref{eq:rmp_continuous_exact} to obtain
    $(\BFx_1^{(t)},\BFtheta^{(t)})$.\;

    Solve the approximate separation problem
    \eqref{eq:violation_subproblem_continuous_approx}
    to obtain $V'^{(t)}$ and a maximizing vertex
    $\widehat{\BFs}^{(t)}$.\;

    \eIf{$V'^{(t)}>\varepsilon_{\mathrm{disc}}$}{
        $\BFs^\star\leftarrow\widehat{\BFs}^{(t)}$.\;
    }{
        Solve the exact separation problem
        \eqref{eq:violation_subproblem_continuous}
        to obtain $V^{(t)}$ and a maximizing vertex
        $\BFs^\star$.\;

        \If{$V^{(t)}\leq\varepsilon_{\mathrm{obj}}$}{
            Compress $\mathcal{S}_t$ to $\mathcal{S}_{\mathrm{obj}}$ by deleting vertices until no further deletion preserves the RMP value.\;
            \KwRet{$(\BFx_1^{(t)},\BFtheta^{(t)},
            \mathcal{S}_{\mathrm{obj}})$}\;
        }
    }

    $\mathcal{S}_{t+1}\leftarrow
    \mathcal{S}_t\cup\{\BFs^\star\}$;
    $t\leftarrow t+1$.\;
}
\end{algorithm}

% ----------- Original --------------
% Because the extreme-point set $\mathcal{S}$ is finite, EPG-O terminates after finitely many iterations. At zero objective tolerance, exact separation certifies an optimal solution of the DRO epigraph \eqref{eq:two_stage_reformulation_extreme_points_continuous}. EPG-O usually generates more than $p$ vertices before termination, so we can compress the set based on Theorem \ref{thm:size_extreme_point_set}. The following theorem formalizes these statements. We admit that the bilinear violation subproblem \eqref{eq:violation_subproblem_continuous} is difficult in general and solving it to optimality may make the algorithm computationally expensive.
% ----------- Revised ---------------
Because $\mathcal{S}$ is finite, exact EPG-O terminates finitely. It may generate more than $p$ vertices, so Theorem~\ref{thm:size_extreme_point_set} supports post-solve compression. We admit that the exact bilinear violation problem \eqref{eq:violation_subproblem_continuous} is difficult and may be computationally expensive.
\begin{theorem}[Objective exactness of EPG-O]
    \label{thm:continuous_convergence}
    Suppose the fixed-recourse setting and Assumption~\ref{asmp:support_reduction_regular} hold, every RMP and compression test is solved exactly, the exact separator returns a maximizing extreme point, and $\varepsilon_{\mathrm{obj}}=0$. Algorithm~\ref{alg:continuous_iterative} terminates finitely. Its output $(\BFx_1^{(t)},\BFtheta^{(t)})$ is optimal for \eqref{eq:two_stage_reformulation_extreme_points_continuous}, and $|\mathcal{S}_{\mathrm{obj}}|\leq p$.
\end{theorem}

\proof{Proof.}
Every RMP retains $\BFx_1\in\mathcal{X}_1$ and omits only extreme-point constraints, so its value is a lower bound. If $V^{(t)}>0$, attained separation yields a maximizing extreme point $\BFs^*\in\mathcal{S}$ that is not in $\mathcal{S}_t$, because all constraints associated with generated points hold for every $\BFxi$ and $\ell$. Thus, each positive violation adds a new vertex, and the finiteness of $\mathcal{S}$ implies finite termination. At zero-tolerance termination, $V^{(t)}\leq0$ means
\[
    \phi_\ell(\BFtheta^{(t)},\BFxi)\geq
    \BFd(\BFxi,\BFx_1^{(t)})^\top\BFz,\qquad
    \forall\BFxi\in\Xi,\ \BFz\in\mathcal{Z},\ \ell\in[L].
\]
The RMP optimizer is therefore feasible for the full master and is globally optimal. Exact compression preserves this value and leaves an inclusion-minimal set. Applying Theorem~\ref{thm:size_extreme_point_set} to the restricted master indexed by that set yields a value-preserving subset of size at most $p$. Inclusion minimality forces this subset to be the full compressed set, so $|\mathcal{S}_{\mathrm{obj}}|\leq p$.
\hfill \Halmos
\endproof

\paragraph{EPG-V for Value Completion.}
EPG-O identifies only the extreme points needed to certify the optimal objective value. Completing the second-stage recourse value requires additional extreme points. EPG-V follows EPG-O and fixes its certified first-stage decision $\BFx_1^*$. Starting from a nonempty set $\mathcal{S}'\subseteq\mathcal{S}$, EPG-V audits whether the generated pieces are value-complete at $\BFx_1^*$.

Suppose Assumption~\ref{asmp:nondegenerate_policy_recovery} holds. For a current nonempty set $\mathcal{T}\subseteq\mathcal{S}$, let $\mathcal{J}_i=\operatorname{supp}(\BFs_i)$ and define the associated affine basis decision as
\[
    \BFx_{2,i}^*(\BFxi)
    :=
    \BFC_{\mathcal{J}_i}^{-1}
    \BFd_{\mathcal{J}_i}(\BFxi,\BFx_1^*),
    \qquad \BFs_i\in\mathcal{T}.
\]
By the correspondence between dual optimality and primal feasibility, if $\mathcal{T}$ is not value-complete, some basis decision $\BFx_{2,i}^*(\BFxi)$ is infeasible for a realization $\BFxi\in\Xi$. We therefore seek such a witness $\BFxi$ and add the corresponding extreme point $\BFs$ to $\mathcal{T}$ until the generated pieces are value-complete.

The current dominance cell of $\BFs_i$ is
\begin{equation}
    \label{eq:dominance_cell_continuous}
    \widehat{\mathcal{R}}_i(\mathcal{T})
    :=
    \left\{
    \BFxi\in\Xi:
    \BFd(\BFxi,\BFx_1^*)^\top\BFs_i
    \geq
    \BFd(\BFxi,\BFx_1^*)^\top\BFs_j,\ 
    \forall \BFs_j\in\mathcal{T}
    \right\}.
\end{equation}
For each nonempty cell and each $k\notin\mathcal{J}_i$, define the row-wise feasibility audit
\begin{equation}
    \label{eq:feasibility_audit_continuous}
    \Delta_{ik}(\mathcal{T})
    :=
    \max_{\BFxi\in\widehat{\mathcal{R}}_i(\mathcal{T})}
    \left\{
    d_k(\BFxi,\BFx_1^*)
    -
    \BFC_k\BFx_{2,i}^*(\BFxi)
    \right\}.
\end{equation}
Because $\BFd$ and the basis decisions are affine in $\BFxi$, the cells and their audit problems admit linear programming representations when $\Xi$ is polyhedral.
Algorithm~\ref{alg:continuous_policy_completion} summarizes the value completion procedure.
\par
\begin{algorithm}[H]
\caption{EPG-V: Basis-assisted critical-region value completion}
\label{alg:continuous_policy_completion}
\KwIn{A nonempty EPG-O output $\mathcal{S}'\subseteq\mathcal{S}$ and decision $\BFx_1$, and $\varepsilon_{\mathrm{aud}}\geq0$}
Initialize $\mathcal{S}_{\mathrm{val}}\leftarrow\mathcal{S}'$.\;
Construct all nonempty cells $\widehat{\mathcal{R}}_i(\mathcal{S}_{\mathrm{val}})$ and solve their audits $\Delta_{ik}(\mathcal{S}_{\mathrm{val}})$.\;
\While{some $\Delta_{ik}(\mathcal{S}_{\mathrm{val}})>\varepsilon_{\mathrm{aud}}$}{
Choose such a pair $(i,k)$ and a maximizing witness $\widehat{\BFxi}\in\widehat{\mathcal{R}}_i(\mathcal{S}_{\mathrm{val}})$.\;
Solve the recourse-dual LP at $\widehat{\BFxi}$ to obtain an optimal extreme point $\BFs^{\mathrm{new}}$. Then set $\mathcal{S}_{\mathrm{val}}\leftarrow\mathcal{S}_{\mathrm{val}}\cup\{\BFs^{\mathrm{new}}\}$.\;
Reconstruct all nonempty cells $\widehat{\mathcal{R}}_i(\mathcal{S}_{\mathrm{val}})$ and solve their audits $\Delta_{ik}(\mathcal{S}_{\mathrm{val}})$.\;
}
Prune $\mathcal{S}_{\mathrm{val}}$ by deleting extreme points whose dominance cell $\widehat{\mathcal{R}}_i(\mathcal{S}_{\mathrm{val}})$ has empty relative interior in \(\Xi\).\;
\KwRet{The audited and compressed value-piece set $\mathcal{S}_{\mathrm{val}}$}\;
\end{algorithm}
We next prove that Algorithm~\ref{alg:continuous_policy_completion} terminates with a set $\mathcal{S}_{\mathrm{val}}$ that is value-complete at the decision $\BFx_1$.

\begin{theorem}[Finite value completion]
    \label{thm:value_completion_continuous}
    Suppose Assumption~\ref{asmp:nondegenerate_policy_recovery} holds, $\varepsilon_{\mathrm{aud}}=0$, and every audit, recourse-dual LP is solved exactly. Algorithm~\ref{alg:continuous_policy_completion} terminates finitely, and its returned set $\mathcal{S}_{\mathrm{val}}$ is value-complete at $\BFx_1$ in the sense that
    \[
        \max_{\BFs\in\mathcal{S}_{\mathrm{val}}}
        \BFd(\BFxi,\BFx_1)^\top\BFs
        =
        \psi(\BFx_1,\BFxi),
        \qquad
        \forall \BFxi\in\Xi.
    \]
\end{theorem}

\proof{Proof.}
Let $\mathcal{T}$ be the current set. If $\Delta_{ik}(\mathcal{T})>0$ at a maximizing witness $\widehat{\BFxi}$, then the basis decision associated with the dominant piece $\BFs_i$ violates row $k$. If $\BFs_i$ were dual optimal at the witness, complementary slackness and Assumption~\ref{asmp:nondegenerate_policy_recovery} would make this basis decision a primal optimum, contradicting the violation. An optimal recourse-dual extreme point at the witness therefore has a strictly larger value and is not in $\mathcal{T}$. Each positive iteration adds a new member of the finite set $\mathcal{S}$.

At termination, fix $\BFxi\in\Xi$ and choose any maximizing $\BFs_i\in\mathcal{S}_{\mathrm{val}}$. Its nonpositive cell audits make the associated basis decision primal feasible. The binding rows and $\BFC^\top\BFs_i=\BFb$ equate its objective with $\BFd(\BFxi,\BFx_1)^\top\BFs_i$, so primal--dual equality proves value completeness. Post-compression deletion preserves the same stopping condition and conclusion.
\hfill \Halmos
\endproof

Under Assumption~\ref{asmp:nondegenerate_policy_recovery}, value completion also implies policy coverage. When Algorithm~\ref{alg:continuous_policy_completion} reaches value completion, the induced policy is feasible for every $\BFxi\in\Xi$ and attains the optimal recourse value. Thus, the policy-recovery conclusion of Proposition~\ref{prop:optimal_decision_rule_continuous} applies to $\mathcal{S}_{\mathrm{val}}$. The basis formula \eqref{eq:optimal_decision_rule_continuous}, with $i_{\BFx_1}^{\mathcal{S}_{\mathrm{val}}}(\BFxi)$ as the selector, gives the resulting optimal policy.

\paragraph{Illustrative example.}
Consider the one-dimensional MAD ambiguity set with $\Xi=[0,1]$, $\mu=1/2$, and $\sigma=1/4$. Suppress the first-stage decision and define
\begin{equation*}
    % \label{eq:continuous_epg_example}
    \psi(\xi)=\min_{x_2}\{x_2\mid x_2\geq q_i(\xi),\ i=1,\ldots,4\},
\end{equation*}
where
\(
    q_1(\xi)=1, q_2(\xi)=3\xi-1,
    q_3(\xi)=\xi+\frac{3}{4}, q_4(\xi)=\frac{3}{2}-3\xi.
\)
The associated dual feasible region is the simplex
\[
    \mathcal{Z}
    =
    \left\{
    \BFz\in\mathbb{R}_+^4
    \ \middle|\ 
    \sum_{i=1}^4 z_i=1
    \right\}.
\]
The four vertices are the unit vectors $\BFe_1,\ldots,\BFe_4$, and the dual objective at $\BFe_i$ equals $q_i(\xi)$. Thus, selecting a dual vertex in this example adds exactly one of the four affine functions to the generated recourse approximation.

Each piece corresponds to an extreme point $\BFe_i$ of the simplex dual region. The optimal MAD majorant is $1+\xi/2+|\xi-1/2|$ and is tight at $\xi = 0, 1/2, 1$, corresponding to extreme points $\BFe_4, \BFe_3$ and $\BFe_2$. Therefore, the three vertices $\BFe_2,\BFe_3,\BFe_4$ form an objective certificate. The fourth piece $q_1$, indexed by $\BFe_1$, is active on $[1/6,1/4]$ and is required for value completion. The objective certificate uses three pieces, whereas pointwise value completion uses all four.

The first three panels of Figure~\ref{fig:continuous_epg_example} trace EPG-O, and the final panel traces EPG-V.
\begin{figure}[htbp]
    \centering
    \includegraphics[width=0.24\textwidth]{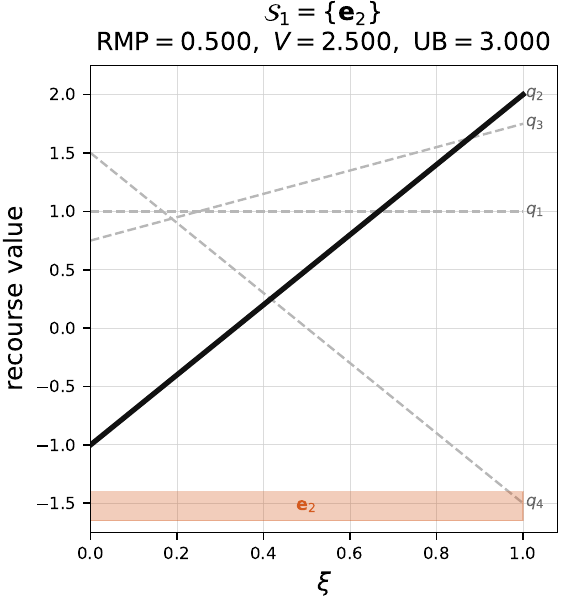}
    \includegraphics[width=0.24\textwidth]{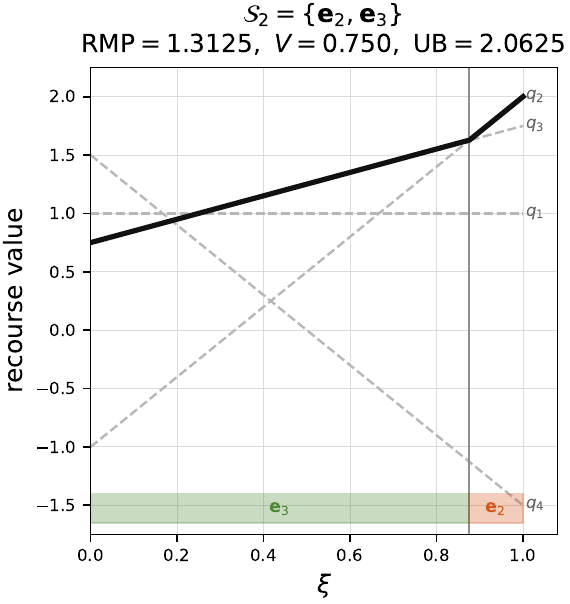}
    \includegraphics[width=0.24\textwidth]{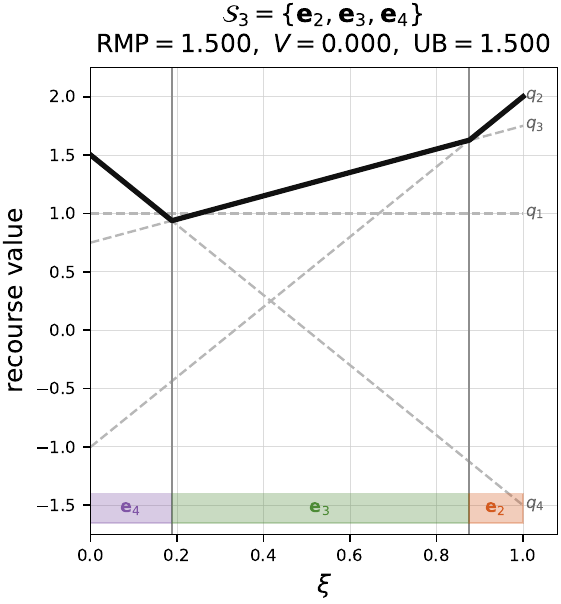}
    \includegraphics[width=0.24\textwidth]{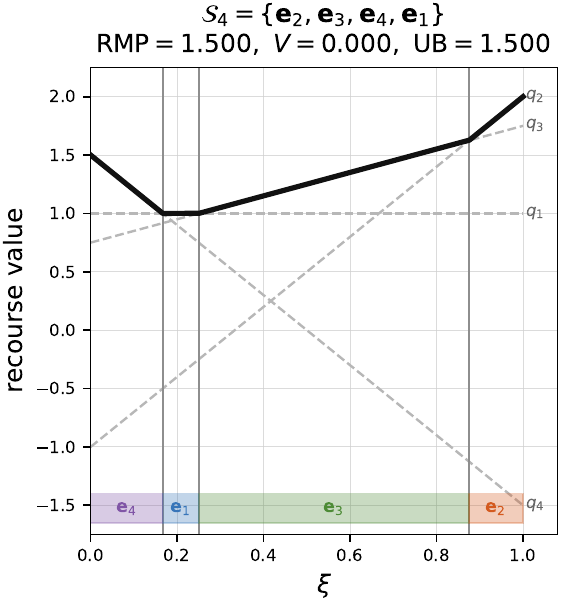}
    \caption{EPG process for the illustrative example. Gray dashed lines show all four affine value pieces, whereas the black line shows the value generated by the current extreme-point set. The colored intervals indicate the regions on which its pieces are active.}
    \label{fig:continuous_epg_example}
    \vspace{-20pt}
\end{figure}
Starting from $\{\BFe_2\}$, EPG-O adds $\BFe_3$ and then $\BFe_4$. At $\mathcal{S}_3=\{\BFe_2,\BFe_3,\BFe_4\}$, the exact objective violation is zero, and $\mathrm{RMP}_3=\mathrm{UB}_3=1.5$. EPG-V then adds $\BFe_1$, leaving the objective bounds unchanged while completing the value on $[1/6,1/4]$. Under Assumption~\ref{asmp:nondegenerate_policy_recovery}, the associated bases yield the optimal decision rule
\[
    x_2^*(\xi)=
    \begin{cases}
        \frac{3}{2}-3\xi, & 0\leq \xi\leq \frac{1}{6},\\
        1, & \frac{1}{6}\leq \xi\leq \frac{1}{4},\\
        \xi+\frac{3}{4}, & \frac{1}{4}\leq \xi\leq \frac{7}{8},\\
        3\xi-1, & \frac{7}{8}\leq \xi\leq 1.
    \end{cases}
\]

\subsection{Value Completion and Pointwise Optimality of Affine Decision Rules}
A central question in the decision-rule literature is when a prescribed decision rule is exact. Parametric linear programming shows that piecewise-affine optimizer maps are optimal but may require exponentially many pieces \citep{borrelli2003geometric,bental2020tractable}. \cite{georghiou2021optimality} provides a condition involving the partition of the uncertainty set and structural properties of the matrix coefficients that guarantees the optimality of $K$-adaptability affine decision rules.

Our framework approaches the same question from a posteriori perspective. EPG-V provides an exact benchmark for assessing a candidate decision-rule architecture. We show that $\mathcal{S}_{\mathrm{val}}$ can provide a natural optimality guarantee for $K$-adaptability affine decision rules and affine decision rules with lifted uncertainty.

\subsubsection{Optimality of $K$-adaptability Affine Decision Rules}
% ----------- Original --------------
% EPG-V identifies all extreme points that can recover the exact recourse value function. The dominance cells in equation \eqref{eq:dominance_cell_continuous} correspond to the uncertainty-set partition condition required by \cite{georghiou2021optimality}, which implies the optimality of $K$-adaptability affine decision rules.
% ----------- Revised ---------------
EPG-V identifies a sufficient set of extreme points for recovering the exact recourse value. The dominance cells in equation \eqref{eq:dominance_cell_continuous} relate to the uncertainty-set partition condition required by \cite{georghiou2021optimality}. Hence, this a posteriori certificate is complementary to the structural sufficient conditions of \cite{georghiou2021optimality}.

To state this result formally, consider the K-adaptability formulation in \cite{georghiou2021optimality}.
\begin{equation}
    \label{eq:K-adaptable_affine_decision_rule}
    \begin{aligned}
        \min_{\BFx_1, \BFx_2^k(\BFxi)}\quad
        & \BFc^\top\BFx_1+\sup_{\mathbb{P}\in\mathcal{P}}\mathbb{E}_{\mathbb{P}}[\psi(\BFx_1,\{\BFx_2^k(\BFxi)\}_{k=1}^{K}, \BFxi)]\\
        \text{s.t.}\quad & \BFA\BFx_1\geq\BFa,\quad \BFx_2^k(\BFxi): \Xi\overset{\mathrm{affine}}{\longrightarrow}\mathbb{R}^{N_2}, \quad k=1,\ldots,K,
    \end{aligned}
\end{equation}
where \(\psi(\BFx_1,\{\BFx_2^k(\BFxi)\}_{k=1}^{K}, \BFxi) = \min_{k \in [K]} \left\{ \BFb^\top\BFx_2^k(\BFxi) \mid \BFC\BFx_2^k(\BFxi)\geq \BFd(\BFxi, \BFx_1) \right\} \) denotes the second-stage value function under the $K$-adaptability affine decision rules. The following proposition gives the resulting sufficient condition for the optimality of $K$-adaptability affine decision rules.

\begin{proposition}[Optimality of $K$-adaptability affine decision rules]
    \label{prop:optimality_k_adaptable}
    Suppose Assumption~\ref{asmp:nondegenerate_policy_recovery} holds, $\BFd(\BFxi,\BFx_1)$ is affine in $\BFxi$, $\Xi$ is polyhedral, and EPG-V satisfies Theorem~\ref{thm:value_completion_continuous}. If $K \geq |\mathcal{S}_{\mathrm{val}}|$, formulation \eqref{eq:K-adaptable_affine_decision_rule} solves \eqref{eq:two_stage_reformulation_extreme_points_continuous} exactly.
\end{proposition}
\proof{Proof.}
% ----------- Original --------------
% Theorem~\ref{thm:value_completion_continuous} shows that \(
%     \max_{\BFs\in\mathcal{S}_{\mathrm{val}}}
%     \BFd(\BFxi,\BFx_1^*)^\top\BFs
%     =
%     \psi(\BFx_1^*,\BFxi),
%     \forall \BFxi\in\Xi
% \), where $\psi(\BFx_1^*,\BFxi)$ is the globally optimal recourse value function.
% Under Assumption \ref{asmp:nondegenerate_policy_recovery}, each extreme point $\BFs_i\in\mathcal{S}_{\mathrm{val}}$ has a dominance cell $\widehat{\mathcal{R}}_i(\mathcal{S}_{\mathrm{val}})$ and an associated basis $\mathcal{J}_i$. The corresponding affine basis decision
% $$
%     \BFx_{2}^*(\BFxi)
%     :=
%     \BFC_{\mathcal{J}_i}^{-1}
%     \BFd_{\mathcal{J}_i}(\BFxi,\BFx_1^*),
%     \ \text{if } \BFxi \in \widehat{\mathcal{R}}_i(\mathcal{S}_{\mathrm{val}})
% $$
% is primal optimal. Because $K \geq |\mathcal{S}_{\mathrm{val}}|$, the solution $(\BFx_1^*,\BFx_2^*(\cdot))$ is feasible under the $K$-adaptability affine decision rule \eqref{eq:K-adaptable_affine_decision_rule}. This rule therefore attains the optimal value of \eqref{eq:two_stage_reformulation_extreme_points_continuous}.
% ----------- Revised ---------------
Index $\mathcal{S}_{\mathrm{val}}=\{\BFs_1,\ldots,\BFs_m\}$ and define the globally affine candidates
\(
\BFx_{2,i}(\BFxi):=\BFC_{\mathcal{J}_i}^{-1}\BFd_{\mathcal{J}_i}(\BFxi,\BFx_1^*)
\)
for $i\in[m]$. For every $\BFxi$, choose $i=i_{\BFx_1^*}^{\mathcal{S}_{\mathrm{val}}}(\BFxi)$. Theorem~\ref{thm:value_completion_continuous} and Assumption~\ref{asmp:nondegenerate_policy_recovery} make this candidate feasible and optimal. Thus, the $m$ global affine candidates attain the unrestricted optimum. Because $K \geq |\mathcal{S}_{\mathrm{val}}|$, the globally optimal solution $(\BFx_1^*,\BFx_2^*(\cdot))$ is feasible under the $K$-adaptability affine decision rule \eqref{eq:K-adaptable_affine_decision_rule}.
\hfill\Halmos
\endproof

Proposition~\ref{prop:optimality_k_adaptable} gives an a posteriori sufficient condition for $K$-adaptability affine rules. When $|\mathcal S_{\mathrm{val}}|=1$, one affine rule is optimal. The number of pieces needed only for the exact objective may be smaller than $|\mathcal S_{\mathrm{val}}|$.

This solution-dependent certificate complements the existence of structural conditions such as those in \cite{georghiou2021optimality}. Proposition~\ref{prop:optimality_k_adaptable} also gives an upper bound on the number of pieces sufficient for exactness, complementing related bounds such as \cite{kurtz2026bounding}. 
However, the evidence in Section~\ref{sec:numerical_experiments} reveals that the number of pieces needed for value completion can substantially exceed the objective-support count.
% This gap motivates an a posteriori test of whether a prescribed architecture represents the completed policy compactly.

\subsubsection{Optimality of Lifted Decision Rules}
Lifting methods prescribe a feature map and optimize a decision rule that is linear in the lifted uncertainty \citep{georghiou2015generalized}. Let $\bm{\Phi}:\Xi\to\mathbb{R}^{N_{\mathrm{lift}}}$ be a lifting map that may include the original uncertainty and nonlinear features. A lifted decision rule has the form
\[
    \BFx_2(\BFxi)=\bm{Y}\bm{\Phi}(\BFxi),
\]
where $\bm{Y}\in\mathbb{R}^{N_2\times N_{\mathrm{lift}}}$. The value-complete set $\mathcal{S}_{\mathrm{val}}$ provides an exact pointwise target for assessing this prescribed architecture.

\begin{proposition}[A posteriori exactness of a lifted decision rule]
    \label{prop:optimality_lifted_decision_rule}
    Let $\BFx_1^*$ be certified by EPG-O, and let $\mathcal{S}_{\mathrm{val}}$ be returned by EPG-V under Theorem~\ref{thm:value_completion_continuous}. For a fixed feature map $\bm{\Phi}$, the lifted policy class contains a pointwise-optimal recourse rule if and only if there exists a matrix $\bm{Y}$ such that, for every $\BFxi\in\Xi$,
    \[
        \BFC\bm{Y}\bm{\Phi}(\BFxi)
        \geq
        \BFd(\BFxi,\BFx_1^*),
        \qquad
        \BFb^\top\bm{Y}\bm{\Phi}(\BFxi)
        =
        \max_{\BFs\in\mathcal{S}_{\mathrm{val}}}
        \BFd(\BFxi,\BFx_1^*)^\top\BFs.
    \]
    When such a matrix $\BFY$ exists, the lifted decision rule solves \eqref{eq:two_stage_reformulation_extreme_points_continuous} exactly.
\end{proposition}
\proof{Proof.}
% ----------- Original --------------
% By Theorem~\ref{thm:value_completion_continuous}, the maximum over $\mathcal{S}_{\mathrm{val}}$ equals $\psi(\BFx_1^*,\BFxi)$ for every $\BFxi\in\Xi$. The first condition makes $\bm{Y}\bm{\Phi}(\BFxi)$ primal feasible, while the second makes its cost equal to the optimal recourse value. Hence, primal--dual equality proves pointwise optimality. Together with the EPG-O certificate for $\BFx_1^*$, this policy attains the unrestricted two-stage DRO optimum. Conversely, every pointwise-optimal lifted rule is primal feasible and has cost $\psi(\BFx_1^*,\BFxi)$ at every realization, so both conditions are necessary.
% ----------- Revised ---------------
By Theorem~\ref{thm:value_completion_continuous}, the maximum over $\mathcal{S}_{\mathrm{val}}$ equals $\psi(\BFx_1^*,\BFxi)$ for every $\BFxi\in\Xi$. The first condition makes $\bm{Y}\bm{\Phi}(\BFxi)$ primal feasible, and the second makes its cost equal to the optimal recourse value. Primal--dual equality therefore proves pointwise optimality. Together with the EPG-O certificate for $\BFx_1^*$, this policy attains the unrestricted two-stage DRO optimum. Conversely, every pointwise-optimal lifted rule is primal feasible and has cost $\psi(\BFx_1^*,\BFxi)$ at every realization, so both conditions are necessary.
\hfill\Halmos
\endproof

Proposition~\ref{prop:optimality_lifted_decision_rule} gives an a posteriori exactness characterization and can guide lifting design. The feature map must represent the piecewise structure. For example, the hinge map $\Phi(\xi)=\left(1, \xi, \left[\xi - \frac{1}{6}\right]_+,\left[\xi - \frac{1}{4}\right]_+,\left[\xi - \frac{7}{8}\right]_+\right)^\top$ represents the optimal rule in Figure~\ref{fig:continuous_epg_example}, but finding the turning points ex ante remains difficult.

\section{Extensions}
We briefly discuss two extensions. The first addresses dual degeneracy and rank-deficient recourse, and the second extends the current fixed-recourse framework to random recourse.

\subsection{Dual Degeneracy and Rank-Deficient Recourse}
\label{sec:degenerate_policy_recovery}
Assumption~\ref{asmp:nondegenerate_policy_recovery} ensures that each extreme point corresponds to a unique, well-defined policy piece. Nondegeneracy gives each extreme point a unique compatible basis, whereas full column rank ensures that the inverse of the basis matrix exists. If either condition fails, reconstructing a primal policy from the dual extreme points may not be unique or well-defined. We now discuss how to handle dual degeneracy and rank-deficient recourse.

\paragraph{Full-rank but degenerate extreme points.}
Suppose $\operatorname{rank}(\BFC)=N_2$, and define $\mathcal{K}_i=\operatorname{supp}(\BFs_i)$. If $|\mathcal{K}_i|<N_2$, then the extreme point $\BFs_i$ is degenerate. Complete $\mathcal{K}_i$ to a row index set $\mathcal{J}$ belonging to
\[
    \mathfrak{J}_i(\BFxi; \BFx_1) = \left\{\mathcal{J} \subseteq [M] \mid
        \mathcal{K}_i\subseteq \mathcal{J},
        |\mathcal{J}|=N_2,
        \det(\BFC_{\mathcal{J}})\neq0,
        \BFC\BFC_{\mathcal{J}}^{-1}\BFd_{\mathcal{J}}(\BFxi,\BFx_1)\geq\BFd(\BFxi,\BFx_1)
    \right\}.
\]
The set $\mathfrak{J}_i(\BFxi; \BFx_1)$ defines the index sets such that $\BFC_{\mathcal{J}}$ is nonsingular and the associated basis decision is feasible. Let $\mathcal{J}^c = [M] \setminus \mathcal{J}$ denote the complement of $\mathcal{J}$ in $[M]$, so $(\BFs_i)_{\mathcal{J}^c}=\BFzero$. The following equations then hold.
\[
    \BFC_{\mathcal{J}}^\top(\BFs_i)_{\mathcal{J}}
    =
    \BFC^\top\BFs_i
    =
    \BFb,
    \qquad
    \BFd^\top(\BFxi,\BFx_1)\BFs_i = \BFd_{\mathcal{J}}^\top(\BFxi,\BFx_1)(\BFs_i)_{\mathcal{J}} = \BFb^\top\BFC_{\mathcal{J}}^{-1}\BFd_{\mathcal{J}}(\BFxi,\BFx_1).
\]
Hence, any $\mathcal{J}\in\mathfrak{J}_i(\BFxi;\BFx_1)$ gives a compatible decision when $\BFs_i$ is dual optimal. Because the basis may vary with $\BFxi$, this is a pointwise reconstruction. An explicit finite policy requires refining each dominance cell into regions with a compatible basis.

\paragraph{Rank-deficient recourse.}
We next consider the case in which $\BFC$ is rank-deficient. Let $r=\operatorname{rank}(\BFC)<N_2$, and choose $\BFU\in\mathbb{R}^{N_2\times r}$ with orthonormal columns that span $\operatorname{range}(\BFC^\top)$. Feasibility and finiteness of the recourse problem imply that $\BFb\in\operatorname{range}(\BFC^\top)$. Set
\[
    \overline{\BFC}=\BFC\BFU,
    \qquad
    \overline{\BFb}=\BFU^\top\BFb.
\]
Apply the change of variables $\BFx_2=\BFU\BFy+\BFv$, where $\BFv\in\ker(\BFC)$. As $\BFy$ ranges over $\mathbb{R}^r$, $\BFU \BFy$ spans $\operatorname{range}(\BFC^\top)$, so $\BFU\BFy+\BFv$ spans the entire space $\mathbb{R}^{N_2}$. Because $\BFC\BFv=\BFzero$ and $\BFb^\top\BFv=0$, the recourse problem has the following equivalent reduced-coordinate representation.
\[
    \min_{\BFy\in\mathbb{R}^r}
    \left\{
    \overline{\BFb}^{\top}\BFy:
    \overline{\BFC}\BFy\geq\BFd(\BFxi,\BFx_1)
    \right\} = \min_{\BFx_2\in\mathbb{R}^{N_2}}
    \left\{
    \BFb^{\top}\BFx_2:
    \BFC\BFx_2\geq\BFd(\BFxi,\BFx_1)
    \right\}.
\]
Because $\operatorname{rank}(\overline{\BFC}) = r$, the matrix $\overline{\BFC}$ has full column rank, and the preceding analysis applies directly.

Applying the preceding basis reconstruction in the reduced coordinates yields $r$ rows indexed by $\mathcal{J}$ for which $\BFC_{\mathcal{J}}\BFU$ is nonsingular. The corresponding reduced decision is $\overline{\BFC}_{\mathcal{J}}^{-1}\BFd_{\mathcal{J}}$. Choosing the canonical element $\BFv=\BFzero$ of the kernel of $\BFC$ yields
\[
    \BFx_2
    =
    \BFU\overline{\BFC}_{\mathcal{J}}^{-1}\BFd_{\mathcal{J}}(\BFxi,\BFx_1).
\]
Thus, the dual value representation remains valid after rank reduction. Applying the same analysis to the reduced-coordinate problem gives a finite set of extreme points and a corresponding decision policy.

\subsection{Random Recourse and Nonlinear Pieces}
\label{sec:structured_random_recourse}
The fixed-recourse theory relies on a dual feasible region that is independent of the realization. Under general random recourse, the primal and dual problems are
\[
    \begin{aligned}
    \psi(\BFx_1,\BFxi)
    &=\min_{\BFx_2}\left\{
    \BFb(\BFxi)^\top\BFx_2:
    \BFC(\BFxi)\BFx_2\geq\BFd(\BFxi,\BFx_1)
    \right\},\\
    D_\psi(\BFx_1,\BFxi)
    &=\max_{\BFz\in\mathcal{Z}(\BFxi)}
    \BFd(\BFxi,\BFx_1)^\top\BFz,
    \qquad
    \mathcal{Z}(\BFxi)
    :=\{\BFz\geq\BFzero:\BFC(\BFxi)^\top\BFz=\BFb(\BFxi)\}.
    \end{aligned}
\]
For each $\BFx_1$ under consideration and every $\BFxi\in\Xi$, suppose $\mathcal{Z}(\BFxi)$ is nonempty, $\BFC(\BFxi)$ has full column rank, relatively complete recourse holds, and the recourse problem attains a finite value.

Although the dual region $\mathcal{Z}(\BFxi)$ varies with $\BFxi$, its candidate row bases belong to the fixed finite family
\[
    \mathfrak J:=\{\mathcal J\subseteq[M]:|\mathcal J|=N_2\}.
\]
% ----------- Original --------------
% For $\mathcal J\in\mathfrak J$, let $\mathcal J^c=[M]\setminus\mathcal J$. Then, the extreme point associated with $\mathcal J$ is $\BFs_{\mathcal J}(\BFxi)$ with 
% ----------- Revised ---------------
For $\mathcal J\in\mathfrak J$, let $\mathcal J^c=[M]\setminus\mathcal J$. The extreme point associated with $\mathcal J$ is $\BFs_{\mathcal J}(\BFxi)$, where
\(
    (\BFs_{\mathcal J}(\BFxi))_{\mathcal J}
    =(\BFC_{\mathcal J}(\BFxi)^\top)^{-1}\BFb(\BFxi)
    \text{ and } 
    (\BFs_{\mathcal J}(\BFxi))_{\mathcal J^c}:=\BFzero.
\)
The feasibility domain of basis $\mathcal J$ is
\[
    \mathcal D_{\mathcal J}
    :=
    \left\{\BFxi\in\Xi:
    \det(\BFC_{\mathcal J}(\BFxi))\neq0,\ 
    (\BFC_{\mathcal J}(\BFxi)^\top)^{-1}\BFb(\BFxi)\geq\BFzero
    \right\}
\]
The dual problem admits an optimal extreme-point solution, and full column rank ensures that one candidate basis represents this extreme point. The recourse value can be expressed as the maximum over the locally valid bases.
\[
    \psi(\BFx_1,\BFxi)
    =
    \max_{\mathcal J\in\mathfrak J:\,\BFxi\in\mathcal D_{\mathcal J}}
    \BFd(\BFxi,\BFx_1)^\top\BFs_{\mathcal J}(\BFxi).
\]
Each locally valid basis $\mathcal{J}$ also yields the candidate decision
\[
    \BFx_2^{\mathcal J}(\BFxi)
    :=
    \BFC_{\mathcal J}(\BFxi)^{-1}
    \BFd_{\mathcal J}(\BFxi,\BFx_1).
\]
This finite-basis formula is only a structural representation, but it indicates the optimality of the nonlinear recourse value function and the decision-rule structure. Extending the EPG method to nonlinear recourse is not straightforward, because it requires solvable nonlinear master and subproblems over the moving basis values and their local domains. We next consider a special structure that restores a realization-independent dual region.

\paragraph{Change-of-variables recourse.}
Consider the structured case
\(
    \BFC(\BFxi)=\BFC^0\BFT(\BFxi),
    \BFb(\BFxi)=\BFT(\BFxi)^\top\BFb^0,
\)
where $\BFT(\BFxi)$ is nonsingular. The dual equality then reduces to
\[
    \BFC(\BFxi)^\top\BFz=\BFb(\BFxi)
    \quad\Longleftrightarrow\quad
    \BFC^{0\top}\BFz=\BFb^0,
\]
so the dual region $\mathcal Z^0:=\{\BFz\geq\BFzero:\BFC^{0\top}\BFz=\BFb^0\}$ is fixed. For each compatible row-index set $\mathcal J$ of $\BFC^0$, the candidate decision becomes
\[
    \BFx_2^{\mathcal J}(\BFxi)
    =
    \BFT(\BFxi)^{-1}(\BFC_{\mathcal J}^0)^{-1}
    \BFd_{\mathcal J}(\BFxi,\BFx_1).
\]
The fixed dual value representation is therefore recovered, although $\BFT(\BFxi)^{-1}$ can make the primal decision nonlinear. Hence, this structure permits value-oriented EPG.

\section{Numerical Experiments}
\label{sec:numerical_experiments}
We study a continuous fixed-depot relief-supply prepositioning problem adapted from the two-stage emergency-supply setting in \cite{rawls2010prepositioning}. Let $\mathcal{I}$ denote the set of fixed depots, $\mathcal{J}$ the set of demand zones, and $\mathcal{A}\subseteq\mathcal{I}\times\mathcal{J}$ the set of shipment arcs. Before demand is observed, $x_i$ units are reserved at depot $i$. After observing the zone-demand vector $\bm{\xi}$, $q_{ij}$ denotes the shipment on arc $(i,j)$, $e_i$ the emergency procurement at depot $i$, and $s_j$ the outsourced coverage in zone $j$. The model is
\begin{equation}
    \label{eq:relief_supply_dro}
    \begin{aligned}
    \min_{\bm{x}}\quad
        &\sum_{i\in\mathcal{I}}c_i x_i
        +\sup_{\mathbb{P}\in\mathcal{P}}\mathbb{E}_{\mathbb{P}}\left[Q(\bm{x},\bm{\xi})\right]\\
    \text{s.t.}\quad
        &0\leq x_i\leq\bar{x}_i,
        &&i\in\mathcal{I},\\
        &\sum_{i\in\mathcal{I}}x_i\leq B.
    \end{aligned}
\end{equation}
where
\begin{equation}
    \label{eq:relief_supply_recourse}
    \begin{aligned}
    Q(\bm{x},\bm{\xi})=\min_{\bm{q},\bm{e},\bm{s} \geq \BFzero}\quad
        &\sum_{(i,j)\in\mathcal{A}}\tau_{ij}q_{ij}
        +\sum_{i\in\mathcal{I}}g_i e_i
        +\sum_{j\in\mathcal{J}}p_j s_j\\
    \text{s.t.}\quad
        &\sum_{i:(i,j)\in\mathcal{A}}q_{ij}+s_j\geq \xi_j,
        &&j\in\mathcal{J},\\
        &\sum_{j:(i,j)\in\mathcal{A}}q_{ij}\leq x_i+e_i,
        &&i\in\mathcal{I}.
        % &q_{ij}\geq0,
        % &&(i,j)\in\mathcal{A},\\
        % &e_i\geq0,
        % &&i\in\mathcal{I},\\
        % &s_j\geq0,
        % &&j\in\mathcal{J}.
    \end{aligned}
\end{equation}
The demand-coverage and depot-flow constraints retain the fixed-recourse structure, whereas emergency procurement and outsourcing provide complete recourse. Appendix~\ref{app:relief_supply_details} presents the full details.

Medium and Large are structured synthetic spatial networks. Medium has four depots and six zones, whereas Large has five depots and eight zones. Daskin aggregates the official location data for the 48 contiguous US state capitals and Washington, DC into six zones and selects four fixed depot representatives \citep{daskin1997network}. Its geography and 1990 population weights are source-based, whereas its costs and demand uncertainty are calibrated for this study. Each family uses ten paired runs with 50 training samples per run. All methods use one thread and a 1,200-second limit. The comparison uses a validated relative-gap target of 0.01 percent, and other details are deferred in Appendix~\ref{app:relief_supply_details}.

We compare five methods. EPG implements Algorithm~\ref{alg:continuous_iterative}. C\&CG adds scenario-indexed recourse copies following \cite{zeng2013columnconstraint}. Scenario-dual CP maintains scenario-indexed recourse-dual information in the spirit of \cite{duque2022distributionally}. MAD Full Vertex enumerates the vertices of the MAD sign-partition cells and solves the corresponding exact extensive formulation \citep{postek2018robust}. ADR is a standard affine-policy benchmark \citep{ben2004adjustable,kuhn2011primal,bertsimas2019adaptive}.

\begin{table}[htbp]
\centering
\caption{Relief-supply MAD-DRO performance over ten paired realizations}
\label{tab:relief_supply_numerics}
\resizebox{\textwidth}{!}{
\begin{tabular}{lccccccccc}
\hline
& \multicolumn{3}{c}{Medium} & \multicolumn{3}{c}{Daskin} & \multicolumn{3}{c}{Large}\\
\cline{2-4}\cline{5-7}\cline{8-10}
Algorithm
& \begin{tabular}[c]{@{}c@{}}$5\%$ gap reached,\\ median time (s)\end{tabular}
& \begin{tabular}[c]{@{}c@{}}$0.01\%$ gap reached,\\ median time (s)\end{tabular}
& \begin{tabular}[c]{@{}c@{}}Median gap (\%)\\120s / final\end{tabular}
& \begin{tabular}[c]{@{}c@{}}$5\%$ gap reached,\\ median time (s)\end{tabular}
& \begin{tabular}[c]{@{}c@{}}$0.01\%$ gap reached,\\ median time (s)\end{tabular}
& \begin{tabular}[c]{@{}c@{}}Median gap (\%)\\120s / final\end{tabular}
& \begin{tabular}[c]{@{}c@{}}$5\%$ gap reached,\\ median time (s)\end{tabular}
& \begin{tabular}[c]{@{}c@{}}$0.01\%$ gap reached,\\ median time (s)\end{tabular}
& \begin{tabular}[c]{@{}c@{}}Median gap (\%)\\120s / final\end{tabular}\\
\hline
EPG & $10/10,\,22.2$ & $8/10,\,19.0$ & $0.010/0.010$
         & $10/10,\,3.76$ & $9/10,\,76.6$ & $0.010/0.010$
         & $10/10,\,14.1$ & $6/10,\,135.2$ & $0.156/0.010$\\
C\&CG & $7/10,\,55.3$ & $2/10,\,559.3$ & $8.258/1.750$
      & $8/10,\,402.3$ & $1/10,\,1070.8$ & $8.949/0.853$
      & $3/10,\,604.2$ & $0/10,\,\text{--}$ & $9.736/7.356$\\
Scenario-dual CP & $7/10,\,248.6$ & $1/10,\,666.6$ & $9.390/1.626$
                 & $8/10,\,600.2$ & $1/10,\,629.7$ & $7.792/1.981$
                 & $3/10,\,130.9$ & $0/10,\,\text{--}$ & $9.736/7.008$\\
MAD Full Vertex & \multicolumn{2}{c}{$1.73$} & $\text{--} / 0$
             & \multicolumn{2}{c}{$2.25$} & $\text{--} / 0$
             & \multicolumn{2}{c}{$27.67$} & $\text{--} / 0$\\
ADR & \multicolumn{2}{c}{$0.058$} & $\text{--} / 10.52$
    & \multicolumn{2}{c}{$0.065$} & $\text{--} / 17.46$
    & \multicolumn{2}{c}{$0.083$} & $\text{--} / 10.02$\\
\hline
\end{tabular}}
\begin{minipage}{\textwidth}
\footnotesize For the iterative methods (EPG, C\&CG, and Scenario-dual CP), each reaching-time cell reports the number of runs that reach the gap and the corresponding conditional median time; unreached runs are not counted. The medians for the 120-second and final gaps use all ten runs. The MAD Full Vertex and ADR cells report median total times. ADR's final column reports its ex-post objective gap relative to the strict exact reference.
\end{minipage}
\end{table}

The results are shown in Table~\ref{tab:relief_supply_numerics}.    
Among the iterative methods, EPG has the lowest reported conditional median target time and the highest or tied-highest attainment count for every family and target. A key observed difference between EPG and the methods that iterate over uncertainty scenarios is that EPG usually requires fewer iterations. One possible explanation is that its fixed dual vertices remain reusable as the first-stage decision changes, whereas generated worst-case scenarios may become inactive because they vary with the first-stage decision more significantly. 

ADR has the shortest runtimes, but its median ex-post objective gaps are 10.52, 17.46, and 10.02 percent. MAD Full Vertex solves the largest reported instance in 27.67 seconds, but its vertex enumeration grows exponentially and therefore requires substantial memory as problem size increases. In the controlled comparison in Appendix~\ref{app:synthetic_scalability}, EPG remains viable and returns a solution with a small gap at larger reported uncertainty dimensions, whereas MAD Full Vertex reaches the memory or time limit.

\paragraph{Value completion.}
Among the 23 Medium, Large, and Daskin cases in which EPG-O reaches the $0.01\%$ gap, the median completion times for Medium, Large, and Daskin are 3.8, 57.6, and 39.1 seconds. The mean effective objective-support and value-completion counts are 8.4 and 74.9. The largest value-completion output occurs for a Large case, with 10 objective-support points and 223 value pieces. These counts show that value completion may retain many more pieces than objective certification. Appendix~\ref{app:relief_supply_details} reports more detailed results.
 
\section{Conclusion}
This paper connects iterative generation and decision rules through the dual envelope of fixed recourse. Dual extreme points define value pieces, and compatible primal bases can turn completed value information into policy pieces. An objective-support set can require fewer pieces than reproducing the recourse value at every realization. In the reported cases meeting the objective-gap target, the effective value-completion sets were larger on average than the effective objective-support sets. A completed value and policy representation may support repeated deployment or an assessment of whether a prescribed decision-rule architecture is sufficiently compact.

The violation subproblem is the main computational bottleneck in the current EPG implementation. 
In the current extensions, we only briefly discuss the random-recourse case and consider a special change-of-variables structure that preserves a fixed dual region. Future work can develop scalable value completion and basis refinement for the random-recourse case, improve exact separation, and evaluate the operational break-even point between offline completion and repeated online recourse optimization.
\clearpage
\bibliographystyle{informs2014}
\bibliography{reference}

\clearpage

\clearpage
\ECSwitch
\renewcommand{\theHequation}{EC.\arabic{equation}}
\renewcommand{\theHtable}{EC.\arabic{table}}
\renewcommand{\theHfigure}{EC.\arabic{figure}}

\begin{APPENDIX}{}
\renewcommand{\theHsection}{EC.\arabic{section}}

\section{Relief-Supply Experimental Details}
\label{app:relief_supply_details}

\paragraph{Formulation and calibration.}
For each zone $j$, define $\mathcal{A}_j=\{i:(i,j)\in\mathcal{A}\}$, and for each depot $i$, define $\mathcal{A}_i=\{j:(i,j)\in\mathcal{A}\}$. Equations~\eqref{eq:relief_supply_dro}--\eqref{eq:relief_supply_recourse}, together with
\begin{equation}
    \label{eq:relief_supply_ambiguity}
    \mathcal{P}=
    \left\{
    \mathbb{P}:
    \mathbb{P}(\bm{\xi}\in\Xi)=1,
    \mathbb{E}_{\mathbb{P}}[\bm{\xi}]=\bm{\mu},
    \mathbb{E}_{\mathbb{P}}[|\bm{\xi}-\bm{\mu}|]\leq\bm{\sigma}
    \right\},
\end{equation}
define the relief model. Let $\bm{d}$ denote base demand and $D=\sum_j d_j$. The polyhedral support is
\begin{equation}
    \label{eq:relief_supply_support}
    \Xi=
    \left\{
    \bm{\xi}:
    0.4d_j\leq\xi_j\leq1.6d_j,\ j\in\mathcal{J},\
    0.75D\leq\sum_{j\in\mathcal{J}}\xi_j\leq1.25D
    \right\}.
\end{equation}
The depot bounds are $\bar{x}_i=0.6\sum_j\mu_j$, and the total prepositioning limit is $B=0.8\sum_j\mu_j$. For each of the 50 observations, independent common and zone-specific Beta$(2,2)$ draws $U_0$ and $U_j$ first generate $d_j(0.5+0.6U_0+0.4U_j)$. The resulting vector is then projected onto \eqref{eq:relief_supply_support}. We set $\bm{\mu}$ to the sample mean and $\bm{\sigma}$ to the componentwise sample MAD.

The transport cost on an active arc is $\tau_{ij}=0.5+4\widetilde{d}_{ij}$, where $\widetilde{d}_{ij}$ denotes distance normalized by the largest depot--zone distance in the instance. Reservation costs repeat $(1.0,1.1,1.2)$ across depots, emergency-procurement costs repeat $(7.0,7.5,8.0)$, and outsourcing costs repeat $(9.0,9.75,10.5,11.25)$ across zones. These calibrated costs make emergency procurement and outsourcing more expensive than planned inventory and regular shipment. They also preserve complete recourse without attributing empirical cost data to the geographic source.

\paragraph{Networks and dimensions.}
For the structured networks, we draw depot and zone coordinates uniformly on the unit square, draw base demands uniformly between 10 and 20, connect each zone to its nearest depots, and repair any unused depot by adding its nearest zone. Small has 3 depots, 4 zones, 8 arcs, 15 recourse variables, and 22 recourse constraints. For the same five dimensions, Medium has $(4,6,18,28,38)$, Daskin also has $(4,6,18,28,38)$, and Large has $(5,8,24,37,50)$. Each family uses ten paired replicate runs, with common instances and training samples across methods.

The Daskin benchmark starts from the official LOCATION49 file, which contains 49 capital-city records and 1990 state-population demand weights, which can be founded in the online appendix of \cite{daskin1997network}. A population-weighted farthest-first rule selects six zone representatives. Each record is assigned to its nearest representative by great-circle distance, and an exhaustive discrete $p$-median calculation selects four fixed depots. Base demand is proportional to aggregated population, whereas costs and uncertainty follow the calibration above. This construction provides source-based geography and population but does not reproduce a disaster instance from \cite{rawls2010prepositioning}.

\paragraph{Method settings and reference construction.}
All computations were performed on a PC equipped with an AMD Ryzen 9 9950X processor, and the algorithms used Gurobi Optimizer 13.0.0 to solve the linear and bilinear programs. All runs use one solver thread and a 1,200-second solver limit. Exact bilinear separation uses the required nonconvex solver option, and the feasibility and optimality tolerances are $10^{-8}$. The iterative exact methods permit at most 1,000 master solves and terminate when completed validation yields a relative gap of at most $10^{-4}$, which is 0.01 percent. MAD Full Vertex enumerates the vertices of every MAD sign cell intersected with \eqref{eq:relief_supply_support}; its total time includes enumeration, formulation construction, solution, and an independent feasibility and objective audit. ADR is solved and audited within the affine-policy class, and its reported ex-post gap compares the audited affine objective with the strict exact reference.

\begin{table}[htbp]
\centering
\caption{Detailed relief-supply computational results}
\label{tab:relief_supply_numerics_detailed}
\scriptsize
\setlength{\tabcolsep}{2pt}
\resizebox{\textwidth}{!}{
\begin{tabular}{llrrrrrrr}
\hline
Family & Method & Strict & \multicolumn{2}{c}{5\% target} & \multicolumn{2}{c}{0.01\% target} & \makecell{Final\\gap} & Runtime \\
 & & runs /10 & reach /10 & time [IQR] & reach /10 & time [IQR] & median [IQR] & median [IQR] \\
\hline
Small & EPG & 9/10 & 10/10 & \makecell[r]{0.392463\\{[0.340192, 0.423352]}} & 10/10 & \makecell[r]{0.471509\\{[0.404332, 0.580018]}} & \makecell[r]{6.34006e-14\\{[4.55457e-14, 1.10975e-13]}} & \makecell[r]{0.47165\\{[0.40445, 0.580132]}} \\
Small & C\&CG & 9/10 & 10/10 & \makecell[r]{0.774941\\{[0.588911, 0.809753]}} & 10/10 & \makecell[r]{1.2016\\{[1.06277, 1.43755]}} & \makecell[r]{6.04687e-14\\{[4.4834e-14, 6.55548e-14]}} & \makecell[r]{1.20169\\{[1.06284, 1.43764]}} \\
Small & \makecell[l]{Scenario-dual\\CP} & 8/10 & 10/10 & \makecell[r]{0.816912\\{[0.754787, 1.00074]}} & 10/10 & \makecell[r]{1.60748\\{[1.51209, 1.76358]}} & \makecell[r]{6.49127e-14\\{[6.13305e-14, 1.05908e-13]}} & \makecell[r]{1.60759\\{[1.51222, 1.76373]}} \\
Small & MAD Full Vertex & 10/10 & 10/10 & \makecell[r]{0.208977\\{[0.185528, 0.238943]}} & 10/10 & \makecell[r]{0.208977\\{[0.185528, 0.238943]}} & \makecell[r]{0\\{[0, 0]}} & \makecell[r]{0.208996\\{[0.185548, 0.238961]}} \\
Small & ADR & -- & -- & -- & -- & -- & \makecell[r]{7.47492\\{[6.88763, 8.94389]}} & \makecell[r]{0.0475422\\{[0.0432791, 0.0507781]}} \\
\hline
Medium & EPG & 0/10 & 10/10 & \makecell[r]{22.2458\\{[10.0054, 103.006]}} & 8/10 & \makecell[r]{18.9517\\{[7.5011, 24.2951]}} & \makecell[r]{0.00999964\\{[0.00999905, 0.00999994]}} & \makecell[r]{22.2459\\{[10.0056, 475.432]}} \\
Medium & C\&CG & 0/10 & 7/10 & \makecell[r]{55.3271\\{[2.88691, 326.911]}} & 2/10 & \makecell[r]{559.347\\{[284.982, 833.711]}} & \makecell[r]{1.74982\\{[0.0889868, 5.55162]}} & \makecell[r]{1201.33\\{[1200.7, 1204.44]}} \\
Medium & \makecell[l]{Scenario-dual\\CP} & 0/10 & 7/10 & \makecell[r]{248.608\\{[4.78944, 528.427]}} & 1/10 & \makecell[r]{666.554\\{[666.554, 666.554]}} & \makecell[r]{1.62617\\{[0.333118, 6.05108]}} & \makecell[r]{1202.72\\{[1201.3, 1206.74]}} \\
Medium & MAD Full Vertex & 10/10 & 10/10 & \makecell[r]{1.72885\\{[1.61221, 2.15304]}} & 10/10 & \makecell[r]{1.72885\\{[1.61221, 2.15304]}} & \makecell[r]{0\\{[0, 0]}} & \makecell[r]{1.72887\\{[1.61223, 2.15306]}} \\
Medium & ADR & -- & -- & -- & -- & -- & \makecell[r]{10.5188\\{[6.4469, 12.3685]}} & \makecell[r]{0.0577616\\{[0.0569037, 0.0699418]}} \\
\hline
Daskin & EPG & 0/10 & 10/10 & \makecell[r]{3.76136\\{[2.65778, 21.4983]}} & 9/10 & \makecell[r]{76.56\\{[45.3474, 183.274]}} & \makecell[r]{0.00999985\\{[0.00999934, 0.00999992]}} & \makecell[r]{85.9628\\{[52.2516, 251.465]}} \\
Daskin & C\&CG & 0/10 & 8/10 & \makecell[r]{402.294\\{[202.881, 792.946]}} & 1/10 & \makecell[r]{1070.81\\{[1070.81, 1070.81]}} & \makecell[r]{0.85283\\{[0.476516, 1.62343]}} & \makecell[r]{1202.27\\{[1200.6, 1204.01]}} \\
Daskin & \makecell[l]{Scenario-dual\\CP} & 0/10 & 8/10 & \makecell[r]{600.181\\{[59.532, 1201.22]}} & 1/10 & \makecell[r]{629.742\\{[629.742, 629.742]}} & \makecell[r]{1.98117\\{[0.596563, 3.29766]}} & \makecell[r]{1201.85\\{[1200.89, 1203.54]}} \\
Daskin & MAD Full Vertex & 10/10 & 10/10 & \makecell[r]{2.25047\\{[2.21054, 2.43888]}} & 10/10 & \makecell[r]{2.25047\\{[2.21054, 2.43888]}} & \makecell[r]{0\\{[0, 1.9839e-14]}} & \makecell[r]{2.25049\\{[2.21057, 2.4389]}} \\
Daskin & ADR & -- & -- & -- & -- & -- & \makecell[r]{17.4586\\{[17.3961, 17.6763]}} & \makecell[r]{0.0645252\\{[0.0607032, 0.0702857]}} \\
\hline
Large & EPG & 0/10 & 10/10 & \makecell[r]{14.1105\\{[7.24726, 760.327]}} & 6/10 & \makecell[r]{135.234\\{[67.0375, 231.282]}} & \makecell[r]{0.01\\{[0.00999992, 0.0388996]}} & \makecell[r]{630.333\\{[119.426, 1201.9]}} \\
Large & C\&CG & 0/10 & 3/10 & \makecell[r]{604.249\\{[314.908, 633.976]}} & 0/10 & {} & \makecell[r]{7.35636\\{[2.41949, 8.68756]}} & \makecell[r]{1202.07\\{[1201.81, 1203.08]}} \\
Large & \makecell[l]{Scenario-dual\\CP} & 0/10 & 3/10 & \makecell[r]{130.911\\{[77.5947, 217.454]}} & 0/10 & {} & \makecell[r]{7.00778\\{[4.27014, 8.35044]}} & \makecell[r]{1202.28\\{[1201.87, 1203.15]}} \\
Large & MAD Full Vertex & 10/10 & 10/10 & \makecell[r]{27.6712\\{[23.0463, 44.8374]}} & 10/10 & \makecell[r]{27.6712\\{[23.0463, 44.8374]}} & \makecell[r]{0\\{[0, 0]}} & \makecell[r]{27.6712\\{[23.0463, 44.8375]}} \\
Large & ADR & -- & -- & -- & -- & -- & \makecell[r]{10.0183\\{[5.44291, 12.5177]}} & \makecell[r]{0.0834402\\{[0.0806314, 0.0892716]}} \\
\hline
\end{tabular}}

\begin{minipage}{\textwidth}
\footnotesize Times are reported in seconds, and gaps are reported as percentages. Target-time summaries are conditional on reaching the target, and an unreached target is left blank rather than assigned the 1,200-second limit. Final-gap summaries use all ten runs. Runtime denotes internal algorithm time. ADR is optimized and audited within the affine-policy class; its Final gap entry reports the ex-post gap relative to the strict exact reference rather than an unrestricted-DRO certificate. Dashes denote quantities that are not applicable.
\end{minipage}
\end{table}

\begin{table}[htbp]
\centering
\caption{Mean EPG-O and EPG-V point counts among cases in which EPG-O reaches the $0.01\%$ gap}
\scriptsize
\setlength{\tabcolsep}{4pt}
{
\begin{tabular}{lrrr}
\hline
Family & \makecell{Cases reaching\\$0.01\%$} & \makecell{Mean EPG-O\\points} & \makecell{Mean EPG-V\\points}\\
\hline
Small  & 10 & 5.3 & 12.5\\
Medium & 8  & 6.4 & 28.1\\
Large  & 6  & 8.8 & 107.0\\
Daskin & 9  & 10.0 & 95.0\\
All    & 33 & 7.5 & 56.0\\
\hline
\end{tabular}}
\begin{minipage}{\textwidth}
\footnotesize Cases are selected solely by whether EPG-O reaches the $0.01\%$ target. EPG-O counts are measured after the compression in Algorithm~\ref{alg:continuous_iterative}, and EPG-V counts are measured after the pruning in Algorithm~\ref{alg:continuous_policy_completion}. The reported sets are algorithm outputs and are not claimed to have globally minimum cardinality.
\end{minipage}
\end{table}

For the 33 cases in which EPG-O reaches the $0.01\%$ gap, the corresponding EPG-V runs at the reference first-stage decisions were certified and passed independent validation. The median value-completion times for Small, Medium, Large, and Daskin were 0.2, 3.8, 57.6, and 39.1 seconds, respectively. The family medians indicate that the LP-based completion phase was computationally practical on the reported instances. One Large case required 10 objective-support points and 223 value-completion points. For this case, the EPG-V completion phase uses 11.3 minutes.

\paragraph{Controlled Scalability of EPG}
\label{app:synthetic_scalability}

The controlled dual family is $\mathcal{Z}=\Delta_q^5$, $q\in\{2,3,4\}$, with 32, 243, and 1,024 extreme points, respectively. We use $N_\xi\in\{9,10,11,12\}$, five paired seeds, box support, and the $3^{N_\xi}$ vertices required by MAD Full Vertex. Unlike the main experiment's 0.01-percent target and 1,200-second limit, this experiment uses a validated 2-percent objective-gap target, a common 180-second wall-clock budget, a 4-GiB memory budget, and one solver thread. Budgeted EPG applies inexpensive approximate discovery before global bilinear validation.

\begin{table}[htbp]
\centering
\caption{Controlled scalability of EPG and MAD Full Vertex}
\label{tab:synthetic_scalability}
\scriptsize
\setlength{\tabcolsep}{2pt}
\renewcommand{\arraystretch}{1.05}
{
\begin{tabular}{rrcrrcr}
\hline
& & \multicolumn{3}{c}{EPG} & \multicolumn{2}{c}{MAD Full Vertex}\\
\cline{3-5}\cline{6-7}
\makecell{Dual\\points} & $N_\xi$ & \makecell{Validated $2\%$\\successes /5} & \makecell{Capped\\time (s)} & \makecell{Retained dual\\points} & \makecell{Validated $2\%$\\successes /5} & \makecell{Capped\\time (s)}\\
\hline
32    & 9  & 5/5 & 0.9   & 7  & 5/5 & 21.6\\
32    & 10 & 5/5 & 1.3   & 8  & 5/5 & 65.7\\
32    & 11 & 5/5 & 1.6   & 7  & 0/5 & 180.0\\
32    & 12 & 5/5 & 4.7   & 11 & 0/5 & 180.0\\
243   & 9  & 5/5 & 5.2   & 12 & 5/5 & 21.4\\
243   & 10 & 5/5 & 29.1  & 14 & 5/5 & 63.9\\
243   & 11 & 5/5 & 26.2  & 15 & 0/5 & 180.0\\
243   & 12 & 3/5 & 129.7 & 18 & 0/5 & 180.0\\
1,024 & 9  & 5/5 & 6.3   & 14 & 5/5 & 21.9\\
1,024 & 10 & 5/5 & 15.9  & 15 & 5/5 & 68.0\\
1,024 & 11 & 3/5 & 123.1 & 21 & 0/5 & 180.0\\
1,024 & 12 & 0/5 & 180.0 & 20 & 0/5 & 180.0\\
\hline
\end{tabular}}
\end{table}

The capped median time assigns 180 seconds to an unsuccessful run, and the reported time and retained-point entries are medians over the five paired runs. Both methods achieve 5/5 successes at $N_\xi=9,10$, whereas MAD Full Vertex achieves 0/5 at $N_\xi=11,12$. EPG remains effective longer, although its success rate weakens with 1,024 dual points. Across cells, EPG's median peak memory remains approximately 161--168 MiB. In contrast, MAD Full Vertex exhibits a sharp increase in memory use around $N_\xi=11$. Seven of the 15 runs hit the 4-GiB limit, seven more time out during optimization after reaching approximately 3.4--3.9 GiB, and the remaining run times out during model construction. At $N_\xi=12$, all 15 MAD Full Vertex runs exhaust the 180-second budget during model construction before optimization, so their lower sampled peak memory does not indicate improved scaling.

Approximate separation generated 488 of the 777 vertices added after initialization, or 62.8 percent. After approximate separation had stalled, 349 iterations reached bilinear validation, which found a residual violated vertex in 289 cases, or 82.8 percent. This rate is conditional on those stalled iterations and is partly induced by the stopping logic. The absolute number of generated vertices generally increases with $N_\xi$ and recourse complexity. Approximation is therefore an inexpensive discovery mechanism, whereas bilinear validation remains necessary for certification and becomes the runtime bottleneck. 

\end{APPENDIX}

\end{document}